\documentclass[11pt,reqno]{amsart}

\usepackage{amsrefs}
\usepackage{amsmath}
\usepackage{amsthm}
\usepackage{graphicx}
\usepackage[dvipsnames]{xcolor}
\usepackage{tikz}
\usepackage{amsfonts}
\usepackage{amssymb}
\usepackage{color}
\usepackage{bm}
\usepackage{hyperref}
\usepackage{enumerate}
\usepackage[margin=1.1in]{geometry}  
\numberwithin{equation}{section}
\newtheorem{theorem}{Theorem}[section]
\newtheorem{lemma}[theorem]{Lemma}
\newtheorem{proposition}[theorem]{Proposition}
\newtheorem{corollary}[theorem]{Corollary}

\theoremstyle{definition}
\newtheorem{example}[theorem]{Example}
\newtheorem{definition}[theorem]{Definition}
\newtheorem{remark}[theorem]{Remark}

\newtheorem{innercustomthm}{Assumption}
\newenvironment{assumption}[1]
  {\renewcommand\theinnercustomthm{#1}\innercustomthm}
  {\endinnercustomthm}

\def\E{{\mathbb E}}
\def\EE{{\mathcal E}}
\def\R{{\mathbb R}}

\def\N{{\mathbb N}}

\def\QQ{{\mathbb Q}}
\def\FF{{\mathbb F}}
\def\P{{\mathcal P}}

\def\K{{\mathcal K}}
\def\X{{\mathcal X}}

\def\G{{\mathcal G}}

\def\Z{{\mathbb Z}}

\def\F{{\mathcal F}}
\def\C{{\mathcal C}}

\newcommand{\half}{{\frac{1}{2}}}

\newcommand{\lan}{\langle}
\newcommand{\ran}{\rangle}

\newcommand{\Emb}{{\mathbb{E}}}

\newcommand{\Rmb}{{\mathbb{R}}}

\newcommand{\Gmc}{{\mathcal{G}}}

\newcommand{\Pmc}{{\mathcal{P}}}

\newcommand{\one}{{\boldsymbol{1}}}

\newcommand{\gmu}{\nu} 
\newcommand{\gX}{Y}
\newcommand{\glambda}{\theta}
\newcommand{\id}{I}

\newcommand{\bK}{\bar{K}}
\newcommand{\bCa}{K}
\newcommand{\bCb}{\bar{K}}

\def\C{{\mathcal C}}
\def\MLset{{\mathcal M}_{\mathrm{init}}}
\def\MRset{{\mathcal M}_{\mathrm{path}}}

\newcommand{\mrfo}{1MRF} 
\newcommand{\mrft}{2MRF}

\title{Locally interacting diffusions as space-time Markov random fields}

\date{\today}
	 	\author[Lacker]{Daniel Lacker}
                \address{Department of Industrial Engineering and Operations Research, Columbia University, New York, NY} 
                \thanks{D. Lacker was partially supported by the Air Force Office of Scientific Research (AFOSR) Grant FA9550-19-1-0291}
	 	\author[Ramanan]{Kavita Ramanan}
                \thanks{K. Ramanan was partially supported by NSF-DMS Grants 1713032 and 1954351, and a Simons Fellowship.} 
	 			\address{Division of Applied Mathematics, Brown University, Providence, RI} 
	 	\author[Wu]{Ruoyu Wu}
	 			\address{Department of Mathematics, Iowa State University, Ames, IA}
	 	 \email{daniel.lacker@columbia.edu, kavita\_ramanan@brown.edu, ruoyu@iastate.edu}

\date{\today}
\begin{document}
\begin{abstract}
We consider a countable system of interacting (possibly non-Markovian) stochastic differential equations driven by independent Brownian motions and indexed by the vertices of a locally finite graph $G = (V,E)$. The drift of the process at each vertex is influenced by  the states of that vertex and its neighbors, and the diffusion coefficient depends on  the state of only that vertex. Such processes arise in a variety of applications including statistical physics, neuroscience, engineering and math finance. Under general conditions on the coefficients, we show that if the initial conditions form a second-order Markov random field on $d$-dimensional Euclidean space, then at any positive time, the collection of histories of the processes at different vertices forms a second-order Markov random field on path space. We also establish a bijection  between (second-order) Gibbs measures on $(\R^d)^V$ (with finite second moments) and a set of space-time (second-order) Gibbs measures on path space, corresponding respectively to the initial law and the law of the solution to the stochastic differential equation. As a corollary, we  establish a Gibbs uniqueness property that shows that for infinite graphs the joint distribution of the paths is completely determined by the initial condition and the specifications, namely the family of conditional distributions on finite vertex sets given the configuration on the complement. Along the way, we establish various approximation and projection results for Markov random fields on locally finite graphs that may be of independent interest.
\end{abstract}

\maketitle

\tableofcontents

\section{Introduction}
\label{sec-intro}

\subsection{Discussion of results}

Given a finite or locally finite infinite graph $G$ with vertex set $V$ and edge set $E$,  and a positive integer
$d$, 
 consider  interacting diffusions that satisfy the following stochastic differential equation (SDE) system:  
\[ dX_v(t) = b_v (t, X_v(t), X_{N_v(G)}(t))\,dt + \sigma_v (t, X_v (t))\,dW_v(t), \quad v \in V, \ \  t \ge 0, 
\] 
where the initial condition  $X(0) = (X_v(0))_{v \in V}$ is distributed according
to some given probability measure on $(\R^d)^V$.
Here, $N_v (G) \subset V$ denotes the neighborhood of $v$ in the graph $G$, 
$(b_v, \sigma_v)_{v \in V}$ are given drift and diffusion coefficients, and
$(W_v)_{v \in V}$ are independent standard
$d$-dimensional Brownian motions.    Such diffusions arise
in a variety of applications,   including statistical 
physics \cites{Der03,RedRoeRus10}, neuroscience
  \cites{LucStan14,Med18}, 
  and systemic risk  \cite{nadtochiy2018mean}.
  Under suitable conditions on the coefficients that guarantee the existence of a unique
weak solution to the SDE,  for any $t > 0$, 
we study the random field  on $\C^V_t$ 
generated by the collection of trajectories $(X_v[t] := (X_v(s))_{s \le t})_{v \in V}$, where $\C_t$ (resp. $\C$) denotes the space of $\R^d$-valued continuous functions on $[0,t]$ (resp. $[0,\infty)$).

Our first set of results  
(Theorems  \ref{pr:conditionalindependence-finitegraph}
and  \ref{th:conditionalindependence-infinitegraph}) 
show that, under modest conditions on the drift and diffusion coefficients that
guarantee a unique weak solution to the SDE system on any locally finite graph, 
if  $(X_v(0))_{v \in V}$ is a second-order Markov random field
on $(\R^d)^V$  
(as specified in Definition \ref{def-MRFs}) then  for each $t > 0$, $(X_v[t])_{v \in V}$
is a second-order Markov random field on $\C_t^V$.    
  In fact, we establish this result for a more general class of SDEs with possibly non-Markovian  dynamics (and potentially infinite memory) defined  in  Section \ref{sec-main}.   
  
  Our next set of results  relate to an interpretation of the law of the
    SDE as a space-time Gibbs measure (see Section \ref{se:gibbs} for precise definitions).
    Specifically, Theorem \ref{th:gibbsuniqueness} establishes 
    a  bijection  between 
    (second-order) Gibbs measures  on $(\R^d)^V$ (with finite second moments) and
    a set of space-time (second-order) Gibbs measures on $\C^V$,  corresponding respectively to
    the initial law and the law of the solution to the SDE.  As a consequence, 
    we deduce a Gibbs uniqueness property, which shows that the law of the SDE system 
     is completely
determined by its initial condition and its specifications,  
namely the family of conditional distributions on finite sets given the configuration on the complement. In particular, together these show (see Corollary \ref{cor:gibbs}) that  
when the initial distribution is the unique second-order 
  Gibbs measure  associated with some specifications on $(\R^d)^V$,
 then  for each $t > 0$,  the law of the SDE system is the unique second-order Gibbs measure associated with corresponding specifications on path space $\C_t^V$.

A key  motivation for our study stems from recent results in 
  \cite{LacRamWu-original,LacRamWu20diffusion} that  show how a second-order Markov property is useful for obtaining an autonomous description of the marginal (local) dynamics
of a particle and its  neighborhood when the underlying graph $G$ is a tree. For this purpose, a stronger \emph{global} Markov property is derived in \cite[Proposition 7.3]{LacRamWu-original} and \cite[Proposition 3.15]{LacRamWu20diffusion} in the setting of an infinite regular tree $G$ (or, more generally, a unimodular Galton-Watson tree) and homogeneous coefficients, $(b_v,\sigma_v)=(b,\sigma)$ for all $v \in V$.
Notably, the characterization of the local dynamics in \cite{LacRamWu-original,LacRamWu20diffusion}  
relies  on  the precise order of the Markov random field (equivalently, range
of interaction of the Gibbs state), and not merely the Gibbs  property. 
  Such an autonomous description, together with the local convergence result \cite[Theorem 3.7]{LacRamWu20convergence}  provides a complete and tractable law of large number result for
  interacting diffusions on certain growing sparse networks.

 In addition, such  characterizations of SDEs in terms of space-time Markov random fields are also of broader interest,
 although past work has mainly focused on the case of interacting diffusions on $\mathbb{Z}^m$, and often with an additional
 gradient structure imposed on the drift. 
The earliest work in this direction appears to be that of Deuschel \cite{Deu87}, who 
considered smooth, uniformly bounded drifts of finite range 
that are locally of  gradient type.  The approach in \cite{Deu87} is based on  estimates of 
Dobrushin's contraction coefficient,   and crucially relies  on the uniform boundedness and 
gradient structure of the drift  (see  \cite[Remark (3.5)i)]{Deu87}).  
This result was later generalized to the case of unbounded  drifts, though still of locally gradient form,  
by Cattiaux, Roelly and Zessin  \cite{CatRoeZes96}. They used
Malliavin calculus, a variational characterization and  an integration-by-parts formula
  (see also \cite{RoeZes93,MinRoeZes00}), all of which exploit the  
  gradient structure of the drift.  
  In \cite{MinRoeZes00}, an alternative cluster expansion method was also used 
  when the gradient system can be viewed as a small perturbation of a free
  field.
  Space-time cluster expansion techniques have also been
    applied to non-gradient systems (although to address the slightly
    different question of the Gibbsian nature  of
  marginals; see also Remark \ref{rem-gibbsprop}), but these only apply  
  for sufficiently small time horizons \cite{RedRoeRus10} or sufficiently
  small interaction strengths \cite{RoeRus14}.  
   An important work that goes beyond the gradient setting is that of Dereudre and Roelly \cite{DereudreRoelly2017},  
  which 
  considers a  system of  interacting one-dimensional SDEs on $\Z^m$ with 
  an adapted (possibly history-dependent) drift 
  that is homogeneous (i.e., the same at all vertices of $\Z^m$) and satisfies a linear growth condition, 
  and with a constant  diffusion coefficient  equal to the identity matrix. 
   Under the restriction that the initial condition is stationary (i.e., shift-invariant on $\Z^{m}$), 
  and has finite specific entropy and  second moment, 
   they  show that the trajectories of the SDE system form  a Gibbs or Markov random field  on $\C_t^{\Z^m}$.
   Their method relies on a Girsanov transformation that crucially relies on the diffusion coefficient not depending
   on the state.    Incorporation of state-dependence in the diffusion
   coefficient is 
   important, for example, to cover 
   basic models of interacting affine diffusions.

 As mentioned above,  we consider the much more general setting 
   of (possibly non-Markovian and time-inhomogeneous)
   interacting diffusions on arbitrary locally finite graphs $G = (V,E)$, 
   with vertex-dependent, adapted, finite-range drift coefficients, 
   and diffusion coefficients that may depend both on the
   vertex and  {\it state} (although we do impose continuity of the drift coefficients since we want uniqueness of
   weak solutions).
   Moreover, we do not assume that the initial conditions are shift-invariant (or, more generally, invariant under automorphisms of the  graph). 
   Thus, we have to  develop new  techniques  to prove our results.
First, to establish the Markov random field property  for SDEs on finite graphs 
(Theorem \ref{pr:conditionalindependence-finitegraph}),  
we
apply a version of the Hammersley-Clifford characterization of Markov random fields on finite graphs.
The result for the infinite graph (Theorem \ref{th:conditionalindependence-infinitegraph})
is  then obtained by a delicate approximation by finite-dimensional systems in a way that preserves the Markov random field property,
along with careful  relative entropy estimates. 
Along the way,   we establish various approximation results for  Markov random fields on locally finite graphs
(see Section \ref{se:2MRFs}) that may be of independent interest.
Thus, our methods are quite different from those of the prior works described above. 
They do not require  any gradient structure of the drift, are not restricted to small perturbations
of a free field and, like \cite{DereudreRoelly2017}, allow for unbounded drifts and non-Markovian
dynamics involving path-dependent coefficients.
However, in constrast to  \cite{DereudreRoelly2017} we consider general graphs and initial conditions. 
Indeed, although the work of  \cite{DereudreRoelly2017} also uses finite-volume approximations, the 
pervasive stationarity assumptions they impose allow them to rely on specific entropy, which cannot be used in our setting,  and they also use 
a growth bound on $\mathbb{Z}^{m}$  (e.g., in deriving equation (30) in the proof of Proposition 3.4 therein),
which does not hold for arbitary  locally finite graphs. 
In addition, while the main goal in \cite{DereudreRoelly2017} is to construct shift-invariant solutions of their
SDE, with the Gibbsian (or Markov random field) description of their process coming as a by-product, our objective
is to identify the extent to which the second-order Markov random field property holds, for which no shift-invariance is required.   
Lastly, we establish a correspondence between Gibbs measures on initial configurations on $(\R^d)^V$ and Gibbs measures on the whole path space $\C^V$ in Theorem \ref{th:gibbsuniqueness} and Corollary \ref{cor:gibbs}, which does not appear in  \cite{DereudreRoelly2017}.

Finally,   we also provide examples (see Section \ref{subs-ceg})  that  demonstrate that the Markov random field property
we establish cannot in general be  significantly strengthened.  Precisely, 
even  on a finite graph with gradient drift, in general the collection of histories
  $(X_v[t])_{v \in V}$ do not form a  first-order Markov random field, nor do the time-$t$ marginals 
    $(X_v(t))_{v \in V}$ exhibit any non-trivial conditional independence structure.
This highlights the natural problem of identifying 
  special classes of systems for which simpler Markov random field properties are preserved, a problem which we do not address but which has attracted considerable attention in certain contexts.
Specifically, in the context of diffusions, the papers \cite{DerRoe04,DerRoe05,RedRoeRus10,RoeRus14,van2009gibbsianness,van2010gibbs} 
have studied the phenomenon of Gibbs-non-Gibbs transitions and the propagation (or lack thereof) of 
the Gibbs property at the level of the time-$t$ marginals,  specifically whether the initial law of $X(0)$ being a Gibbs state on $(\R^d)^V$ 
implies that the marginal  law of  $X(t)$ is also a Gibbs state on $(\R^d)^V$.
  See Remark \ref{rem-gibbsprop} for a more detailed description of these works.

The next section introduces some common notation and basic definitions
   used throughout the paper.
 The main results of the paper are stated in
 Section \ref{sec-main}, with their proofs relegated to Sections \ref{sec-finitegraph}--\ref{se:gibbsuniqueness}.

  \subsection{Notation and basic definitions}
  \label{subs-not}

For any vectors $a, b \in \R^d$, we use $a \cdot b$ or $\langle a, b\rangle$ to denote the inner product. 
In this paper, unless explicitly stated otherwise, a \emph{graph} $G=(V,E)$ always has a finite or countably infinite vertex set, is simple (no self-edges or multi-edges), and is locally finite (i.e., the degree of each vertex is finite).
We abuse notation by writing $v \in G$ to mean $v \in V$.  For a graph $G=(V,E)$ and a vertex $v \in V$, we write $N_v(G) = \{u \in V : (u,v) \in E\}$ for the set of neighbors of $v$ in $G$, noting that this set is empty if $v$ is an isolated vertex. A \emph{rooted} graph $G=(V,E,\o)$ is a graph equipped with a distinguished vertex $\o \in V$, called the \emph{root}.
For two vertices $u,v \in V$, let $d(u,v)$ denote the graph distance, i.e., the length of the shortest path from $u$ to $v$ (with $d(u,u) := 0$).
Also, 
let $\mathrm{diam}(A)$ denote the diameter of a set $A \subset V$; precisely,
$\mathrm{diam}(A) = \sup\{d(u,v) : u,v \in A\}$.
For a subset $A \subset V$, we define the first and second boundaries
\begin{equation}
  \label{bdaries}
\begin{array}{rcl}
\partial_G A &=& \{u \in V \backslash A : (u,v) \in E \text{ for some }v \in A\}, \\
\partial^2_G A &=& \partial_G A \cup \partial_G (A \cup \partial_G A).
\end{array}
\end{equation}
We will often omit the subscript, writing simply $\partial^2A$ in place of $\partial^2_GA$, when the underlying graph $G$ is clear.
A \emph{clique}  in a graph $G$ is a complete subgraph of $G$, i.e., a set $A \subset V$ such that $(u,v) \in E$ for every $u,v \in A$. Equivalently, a clique is a set $A \subset V$ of diameter at most $1$. Define $\mathrm{cl}_1(G)$ to be the set of all cliques of the graph $G$.   
Similarly, we will say that any subset $A \subset V$ with diameter at most $2$ is a
{\em $2$-clique} of the graph $G$ and let $\mathrm{cl}_2(G)$ denote 
the set of $2$-cliques of  $G$. 
Moreover, given a graph $G = (V,E)$, $H = (V_H, E_H)$ is said to be an induced subgraph of
$G$ if $V_H \subset V$ and $E_H = E \cap \{(u,v): u, v \in V_H\}$.

For a set $\X$ and a graph $G= (V,E)$, we may write either $\X^V$ or $\X^G$ for the configuration space
$\{(x_v)_{v \in V}: x_v \in \X \mbox{ for every } v \in V\}$. We make use of a standard notation for configurations on subsets of vertices: For $x=(x_v)_{v \in V} \in \X^V$ and $A \subset V$, we write $x_A$ for the element $x_A=(x_v)_{v \in A}$ of $\X^A$.   
When  $\X$ is a Polish space, we write $\P(\X)$ for the set of Borel probability measures on $\X$, endowed always with the topology of weak convergence.  Given  any measurable space
$\X$, $A \subset V$,  and measure $\gmu \in \P(\X^V)$, $\gmu[A]$ represents the
 restriction of $\gmu$ to the set $\X^A$, that is, the image measure under the restriction
 map $\X^V \ni (x_v)_{v\in V} \mapsto (x_v)_{v \in A} \in \X^A$. 

Fixing $d \in \N$, we let 
$\C = C(\R_+;\R^d)$ denote the path space 
of $\R^d$-valued continuous functions on $\R_+=[0,\infty)$,
  endowed with the topology of uniform convergence on compacts. 
  For $t > 0$, let  $\C_t = C([0,t];\R^d)$ denote its restriction to the time interval $[0,t]$, endowed with the uniform topology.
For $x \in \C$ and $t > 0$ we define $\|x\|_{*,t} := \sup_{s \in [0,t]}|x(s)|$, and let $x[t] = (x(s))_{s \le t} \in \C_t$ denote the restriction of the path $x$ to the time interval  $[0,t]$. 
We assume that $\C$ and $\C_t$ are endowed with their respective Borel $\sigma$-algebras.  
 Also, for any countable set $A$ and probability measure $Q$ on $\C^A$, we write $Q_t$ for the image under $Q$ of the map $\C^A \ni (x_v)_{v \in A} \mapsto (x_v[t])_{v \in A} \in \C_t^A$.
 The  $\sigma$-algebra on a product space will always be the product $\sigma$-algebra, unless
 explicitly stated otherwise.
 Given  $J, m \in \N$,  a measurable function $f:[0,\infty) \times \C^J \mapsto \R^m$, 
    is said to be {\em progressively measurable} if for each $t \ge 0$,  
   $f(t,(\tilde{x}_u)_{u=1, \ldots, J}) = f(t,(y_u)_{u = 1, \ldots, J})$ whenever $x_u[t]=y_u[t]$ for all $u = 1, \ldots, J$.
 
  We end this section by recalling the notion of a (first-order or second-order) Markov random field, which plays a central role in the paper.

\begin{definition}
  \label{def-MRFs}
  Given a measurable space $\X$, and a (possibly infinite)
  locally finite graph $G = (V,E)$, 
  let $(\gX_v)_{v \in V}$ be a random element of $\X^V$ with some distribution $\gmu \in \P(\X^V)$.
We say that $(\gX_v)_{v \in V}$, or equivalently its law $\gmu$, is a \emph{first-order Markov random field} (abbreviated as \mrfo) on $\X^V$ if $\gX_A$ is conditionally independent of $\gX_{(A \cup \partial A)^c}$ given $\gX_{\partial A}$, for every finite set $A \subset V$. 
On the other hand, we say that $(\gX_v)_{v \in V}$, or equivalently its law $\gmu$,
is a \emph{second-order Markov random field} (abbreviated as \mrft)  on $\X^V$ if $\gX_A$ is conditionally independent of $\gX_{(A \cup \partial^2 A)^c}$ given $\gX_{\partial^2 A}$, for every finite set $A \subset V$.
When the space $\X^V$ is clear from the context, we will simply say that $(\gX_v)_{v \in V}$, or equivalently its law $\gmu$, is a \mrfo \ or \mrft.
\end{definition}

\begin{remark}
  \label{rem-MRFs} 
In Definition \ref{def-MRFs}, it is important to stress that  the sets A are required to be finite even when the graph $G$  is infinite.  Allowing infinite sets $A$ results in the stronger \emph{global Markov property}, which we do not study in this paper.   
\end{remark}

\section{Main results}
\label{sec-main}

Given a locally finite graph $G = (V,E)$ with a finite or countably infinite vertex set, 
we are interested in a system of  (possibly non-Markovian) interacting stochastic processes, indexed by the
vertices of the graph, that satisfy an SDE of the form 
\begin{align}
  \label{fingraph-SDE}
dX_v(t) &= b_v(t,X_v,X_{N_v(G)}) \,dt + \sigma_v(t,X_v) \,dW_v(t), \quad v \in V,
\end{align}
where $(W_v)_{v \in V}$ are independent Brownian motions,  and 
the initial law $\mu_0 \in \P((\R^d)^V)$, of $(X_v(0))_{v \in V}$ 
and the coefficients $(b_v, \sigma_v)_{v \in V}$  satisfy                                                                                           the conditions stated
in Assumption \ref{assumption:A} or Assumption \ref{assumption:B} below, depending on whether
the graph is finite or infinite. 
As mentioned in the introduction, our  main results concern the characterization
of the law of the solution to the  SDE \eqref{fingraph-SDE} as a \mrft \ on $\C^V$ (see  Definition \ref{def-MRFs}).

\subsection{The finite graph case}
\label{subs-mainfin}

We first consider the case when $G$ is finite,  and the 
 conditions stated in Assumption \ref{assumption:A} below are satisfied. 
 Recall, from  Section \ref{subs-not}, the definition of $2$-cliques,  the notation for trajectories,
 $x[t] = (x(s))_{s \le t} \in \C_t$ for $x \in \C$, and the notion of a progressively measurable functional.

\begin{assumption}{\textbf{A}} \label{assumption:A} $\ $
  We say that $(G,b, \sigma, \mu_0)$ satisfy Assumption \ref{assumption:A}
  if $G = (V,E)$ is a finite graph and if $b=(b_v)_{v \in V}$, $\sigma = (\sigma_v)_{v \in V}$, and $\mu_0 \in \P((\R^d)^V)$ satisfy the following:  
\begin{enumerate}
\item[(\textbf{A}.1)] There exist $\lambda_v \in \P(\R^d), v \in V$, such that
  the probability  measure $\mu_0$ is absolutely continuous with respect to
  the product measure $\mu_0^* = \prod_{v \in V}\lambda_v \in \P((\R^d)^V)$ and the density  satisfies 
\begin{align}
\frac{d\mu_0}{d\mu_0^*}(x) = \prod_{K \in \mathrm{cl}_2(G)}f_K(x_K), \quad\quad x \in (\R^d)^V, \label{asmp:A1factorization}
\end{align}
for some measurable functions $f_K : (\R^d)^K \rightarrow \R_+$, $K \in \mathrm{cl}_2(G)$, where $\mathrm{cl}_2(G)$ is the set of $2$-cliques of $G$.  In addition, $\mu_0$ has a finite second moment.
\item[(\textbf{A}.2)] For each $v \in V$, the drift $b_v: \R_+ \times \C \times \C^{N_v(G)} \mapsto \R^d$ is progressively measurable.   Moreover,  for each $T \in (0,\infty)$ there exists $C_T < \infty$
  such that
\[
|b_v(t,x,y_{N_v(G)})| \le C_T\left(1 + \|x\|_{*,t} + \sum_{u \in N_v}\|y_u\|_{*,t}\right),
\]
for all $v \in V$, $t \in [0,T]$, $x \in \C$, and $y_{N_v(G)} =(y_u)_{u \in N_v(G)} \in \C^{N_v(G)}$.
\item[(\textbf{A}.3)] The diffusion matrices $\sigma_v : \R_+ \times \C \rightarrow \R^{d\times d}$, $v \in V$,  satisfy the following:
  \begin{enumerate}
  \item[(\textbf{A}.3a)] For each $v \in V$, $\sigma_v$ is bounded, progressively measurable and invertible,
    with bounded inverse. 
  \item[(\textbf{A}.3b)] For each $v \in V$, the following driftless SDE system admits a unique in law weak solution starting from any initial position $x \in \R^d$: 
\begin{align}
dX_v(t) = \sigma_v(t,X_v)\,dW_v(t), \quad X_v(0)=x. \label{eq:canonical_law-individual}
\end{align} 
\end{enumerate}
\end{enumerate}
\end{assumption}

\begin{remark}
A necessary condition for Assumption (\ref{assumption:A}.1) is that $\mu_0$ is a \mrft \ and is absolutely continuous with respect to the product measure $\mu_0^*$; this follows from a form of the Hammersley-Clifford theorem stated in Proposition \ref{th:hammersleyclifford2} below. If the density $d\mu_0/d\mu_0^*$ is strictly positive, then it factorizes as in \eqref{asmp:A1factorization} if and only if $\mu_0$ is a \mrft.
\end{remark}

\begin{remark}
\label{rem-assA3}
If $\sigma_v(t,x)=\widetilde\sigma_v(t,x(t))$ depends only on the current state, not the history, and satisfies the additional continuity condition $\lim_{y \rightarrow x} \sup_{0 \leq s \leq T}|\widetilde\sigma_v(s,y) - \widetilde\sigma_v(s,x)| = 0$ for all $v \in V$, then Assumption (\ref{assumption:A}.3b) holds as a consequence of  Assumption (\ref{assumption:A}.3a) and \cite[Chapter 7]{StroockVaradhan}. 
\end{remark}

The following proposition shows that,
as a simple consequence of Girsanov's theorem,  Assumption \ref{assumption:A} 
guarantees weak existence and uniqueness in law of the SDE system \eqref{fingraph-SDE}. 
Its proof is given in Section \ref{ap-finite}, along the way to proving Theorem \ref{pr:conditionalindependence-finitegraph} below.

\begin{proposition}
  \label{prop-unique-finite}
  When $(G,b, \sigma, \mu_0)$ satisfy Assumption \ref{assumption:A},  
    the SDE system   \eqref{fingraph-SDE} has a weak solution that is unique in law. 
\end{proposition}

We now state our  main result for the SDE system on finite graphs.

\begin{theorem} \label{pr:conditionalindependence-finitegraph}
Suppose  $(G=(V,E),b, \sigma, \mu_0)$ satisfy Assumption \ref{assumption:A}, and let $(X_v)_{v \in V}$ be the unique in law solution of the SDE system \eqref{fingraph-SDE} with  initial law $\mu_0$. Then, for each $t > 0$, $(X_v[t])_{v \in V}$ is a \mrft \ on $\C_t^V$.
Moreover, $(X_v)_{v \in V}$ is a \mrft \ on $\C^V$.
\end{theorem}

The proof of Theorem \ref{pr:conditionalindependence-finitegraph},  given
in Section \ref{ap-finite},  
relies on a certain factorization property (stated in Proposition \ref{th:hammersleyclifford2}) of the density of the law of 
the SDE on finite graphs with respect to a reference measure. 
 Notice that in Assumption \ref{assumption:A}, and throughout the paper, we assume there is no interaction in the diffusion coefficients (i.e., no dependence of $\sigma_v$ on $X_{N_v(G)}$), a restriction made also in the prior works \cite{Deu87,RoeZes93,CatRoeZes96,MinRoeZes00}; the general case seems out of reach of our approach, because the reference measure in the factorization property must crucially be a \emph{product measure}. 
This factorization property is also used in Sections
\ref{se:example-firstorder} and \ref{se:example-marginals} to show that, even when  
the initial states $(X_v(0))_{v \in V}$ are i.i.d.,  for $t > 0$, 
in general $(X_v[t])_{v \in V}$ fails to be a \mrfo,
and the time-$t$ marginals $(X_v(t))_{v \in V}$ can fail to be a 
Markov random field of either first or second order. 
In fact, the counterexamples show that this does not hold even on a finite graph when $\sigma$ is the identity covariance
matrix, and the drift is of gradient type.
This shows that, in a sense, 
Theorem \ref{pr:conditionalindependence-finitegraph} cannot be strengthened.

\subsection{The infinite graph case}
\label{subs-maininfin}

We now consider the SDE system \eqref{fingraph-SDE} in the case when $G$ is an infinite, 
though still locally finite, graph.  
The well-posedness of the SDE system is no longer obvious and in particular 
does not follow from Girsanov's theorem as it did when the graph was finite. 
Indeed, on an infinite graph, when $b_v \equiv 1$ and $\sigma_v \equiv I_d$, $v \in V$, for instance, it is straightforward to argue, using the law of large numbers, that the law of a weak solution of \eqref{fingraph-SDE}  up to some time $t > 0$ and the law of the corresponding drift-free equation are mutually singular.
This necessitates the following additional assumptions compared to Assumption 
\ref{assumption:A}.   Recall from Section \ref{subs-not} that given a measure $\nu$ on ${\mathcal X}^V$ for some Polish space ${\mathcal X}$, and $A \subset V$, $\nu[A]$ denotes the restriction of $\nu$ to $A$.
Also, we use the notation  $\pi_1 \sim \pi_2$ to denote that  the measures $\pi_1$ and $\pi_2$
are equivalent, that is,  mutually absolutely continuous.

\begin{assumption}{\textbf{B}} \label{assumption:B} $\ $
   We say that $(G,b, \sigma, \mu_0)$ satisfy Assumption \ref{assumption:B} if $G=(V, E)$ is a countable locally finite  connected graph and 
  if $b=(b_v)_{v \in V}$, $\sigma = (\sigma_v)_{v \in V}$ and $\mu_0 \in \P((\R^d)^V)$ satisfy the following:     
  \begin{enumerate}
  \item[(\textbf{B}.1)]
    The initial law $\mu_0$ is a \mrft \  on $(\R^d)^V$.  Moreover, there exists a product measure $\mu_0^* = \prod_{v \in V}\lambda_v \in \P((\R^d)^V)$ such that $\mu_0[A] \sim\mu_0^*[A]$ for each finite set $A \subset V$. 
    Further, the initial law $\mu_0$ satisfies 
\begin{align}
  \sup_{v \in V}\int_{(\R^d)^V}|x_v|^2\,\mu_0(dx_V) < \infty. \label{assumption:initialsecondmoment}
  \end{align}
    \item[(\textbf{B}.2)]  
   The drift coefficients $(b_v)_{v \in V}$ satisfy 
   Assumption (\ref{assumption:A}.2), for some constants $(C_T)_{T > 0}$. 
\item[(\textbf{B}.3)]  The diffusion matrices $(\sigma_v)_{v \in V}$ satisfy Assumption (\ref{assumption:A}.3). 
\item[(\textbf{B}.4)] The SDE system \eqref{fingraph-SDE} is unique in law, and this law is denoted by $P = P^{\mu_0} \in \P(\C^V)$.
\end{enumerate}
\end{assumption}

\begin{remark}
	\label{rem-driftless}
	Using Assumption (\ref{assumption:A}.3b) if the graph is finite or Assumption (\ref{assumption:B}.3) if the graph is infinite, we may define for any initial law $\nu \in \P((\R^d)^V)$ the measure $P^{*,\nu} \in \P(\C^V)$ to be the law of the unique weak solution of the SDE system
	\begin{align}
	dX_v(t) = \sigma_v(t,X_v)\,dW_v(t), \quad v \in V, \ \ (X_v(0))_{v \in V} \sim \nu. \label{eq:canonical_law}
	\end{align}
	Note in particular that if we take $\nu=\mu_0^*$, where $\mu_0^*$ is a product measure as in Assumption (\ref{assumption:A}.1) or (\ref{assumption:B}.1), then $P^{*,\mu_0^*}$ too is a product measure.
\end{remark}

We show in Lemma \ref{le:infinitegraphlimit} that existence of a solution to \eqref{fingraph-SDE} 
 follows automatically from Assumptions (\ref{assumption:B}.1--3).  
However, it is worth commenting on the uniqueness condition in  Assumption
(\ref{assumption:B}.4).  The following  proposition shows that a
suitable Lipschitz condition  is enough to guarantee uniqueness;
its proof is standard and hence relegated to  Appendix \ref{ap:uniqueness-infSDE}. 
Recall in the following that $\|x\|_{*,t} = \sup_{s \in [0,t]}|x(s)|$ for $x \in \C$.

\begin{proposition} \label{pr:uniqueness-infiniteSDE}
  Suppose Assumptions (\ref{assumption:B}.1--3)  hold,
  and $(b_v)_{v \in V}$ and $(\sigma_v)_{v\in V}$ are uniformly Lipschitz in  the sense that 
  for each $T > 0$ there exist $\bCa_T, \bCb_T <    \infty$ such that 
 \begin{align*} 
 |b_v(t,x,y_{N_v(G)}) - b_v(t,x',y'_{N_v(G)})| &\le K_T\left(\|x-x'\|_{*,t} + \frac{1}{|N_v(G)|}\sum_{u \in N_v(G)}\|y_u-y'_u\|_{*,t}\right), \\ 
 |\sigma_v(t,x) - \sigma_v(t,x')| &\le \bK_T\|x-x'\|_{*,t}, 
   \end{align*}
 for all $v \in V$, $t \in [0,T]$, $x,x' \in \C$, and $y_{N_v(G)},y'_{N_v(G)} \in \C^{N_v(G)}$.  
Then pathwise uniqueness holds for the SDE system \eqref{fingraph-SDE}.
In particular, Assumption (\ref{assumption:B}.4) holds.
\end{proposition}

We now state our second main result. 

\begin{theorem} \label{th:conditionalindependence-infinitegraph}
Suppose  $(G=(V,E),b, \sigma, \mu_0)$ satisfy Assumption \ref{assumption:B}, and let $(X_v)_{v \in V}$ be the unique in law solution of the SDE system \eqref{fingraph-SDE} with initial law $\mu_0$. Then, for each $t > 0$, $(X_v[t])_{v \in V}$ is a \mrft \ on $\C_t^V$. 
Moreover, $(X_v)_{v \in V}$ is a \mrft \ on $\C^V$.
\end{theorem}

The proof of Theorem \ref{th:conditionalindependence-infinitegraph} is given in
Section \ref{subs-pf-infingraph}, with preparatory results established in Sections \ref{se:2MRFs}
and \ref{subs-tightness}.  
The factorization result used in the finite graph case  is no longer applicable in the infinite graph case,
and thus the proof employs  a completely different approach, involving a rather
subtle approximation argument, which is outlined in Section \ref{subs-approxmeas}.

\subsection{Gibbs measures on path space} \label{se:gibbs}

Our final results interpret our SDE system in the spirit of Gibbs measures, 
for which we introduce the following notation. 
Given a Polish space $\X$, a graph $G=(V,E)$, a random  $\X^V$-valued element $(\gX_v)_{v \in V}$  with
law $\gmu \in \P(\X^V)$, and two disjoint sets $A,B \subset V$, we write 
$\gmu[A \, | \, B]$ to denote a version of the regular conditional law of $\gX_A$ given $\gX_B$. Precisely, we view $\gmu[A \, | \, B]$ as a measurable map (kernel) from $\X^B$ to $\P(\X^A)$. Note that $\gmu$ is a \mrft\ if and only if $\gmu[A \, | \, V \backslash A](x_{V \backslash A}) = \gmu[A \, | \, \partial^2A](x_{\partial^2A})$ for $\gmu$-almost every $x \in \X^V$ and every finite set $A$.
We make use of the following terminology of Gibbs measures (see \cite{georgii2011gibbs} or \cite{rassoul2015course} for further discussion of this classical framework).

  \begin{definition}
    \label{def-Gibbs}
   Given a Polish space $\X$,  graph $G=(V,E)$,  and $\gamma \in \P(\X^V)$, define
  $\G_2(\gamma)$ as the set of 
   \mrft s $\nu \in \P(\X^V)$ such that, for each finite set $A \subset V$, we have $\nu[A] \sim \gamma[A]$ and also $\nu [A \, | \, \partial^2 A] = \gamma[A \, | \, \partial^2 A]$, almost everywhere with respect to $\gamma[\partial^2 A]$.
  \end{definition}
  
  In other words,  ${\mathcal G}_2(\gamma)$ is 
    the set of (second-order, infinite volume) Gibbs  measures  corresponding to the \emph{specification} $\{\gamma[A \, | \, \partial^2A] : A \subset V \ \rm{finite}\}$.
  Note that if $\gamma$ is itself a \mrft\ then $\G_2(\gamma)$ is nonempty, as it contains $\gamma$ itself. Moreover, it is straightforward to check that, if $\gamma$ and $\nu$ are \mrft s, then $\nu \in \G_2(\gamma)$ if and only if $\gamma \in \G_2(\nu)$. 
   Recall that, by Assumption (\ref{assumption:B}.4), the SDE system \eqref{fingraph-SDE} is well-posed starting from any initial distribution. 
 The following bijection result 
 is proved in Section \ref{se:gibbsuniqueness}.

\begin{theorem} \label{th:gibbsuniqueness}
  Suppose $(G, b, \sigma, \mu_0)$ satisfy Assumption \ref{assumption:B}.  Let $P^{\mu_0} \in \P(\C^V)$ 
  be the law of the solution of the SDE
  system  \eqref{fingraph-SDE} with initial law $\mu_0$ 
  and define
  \[  \MLset (\mu_0) := \Big\{ \nu_0 \in \G_2(\mu_0) : \sup_{v \in V}\int_{\R^d}|x_v|^2\,\nu_0(dx) < \infty\Big\}, \]
  and
  \[   \MRset (\mu_0) := \left\{ Q \in \P(\C^V) : Q_t \in \G_2(P^{\mu_0}_t) \ \forall t \ge 0, \ \sup_{v \in V}\int_{\C^V}|x_v(0)|^2Q(dx) < \infty\right\}.
  \]
Then it holds that
\begin{align}
  \MLset(\mu_0)  = \Big\{Q \circ (X_V(0))^{-1} : Q \in \MRset(\mu_0) \}. 
  \label{claim} 
\end{align} 
Moreover, the map $Q \mapsto Q \circ (X_V(0))^{-1}$ defines a bijection between $\MRset(\mu_0)$ and
$\MLset(\mu_0)$.  In particular, if $Q \in \P(\C^V)$ satisfies $Q_t \in \G_2(P^{\mu_0}_t)$ for all $t \ge 0$ and also $Q \circ (X_V(0))^{-1} = \mu_0$, then $Q=P^{\mu_0}$. 
\end{theorem}

In fact, we will show in the proof of Theorem \ref{th:gibbsuniqueness} that the bijection $Q \mapsto Q \circ (X_V(0))^{-1}$ between the sets $\MRset(\mu_0)$ and $\MLset(\mu_0)$
has inverse given by $\nu_0 \mapsto P^{\nu_0}$, where $P^{\nu_0}$ denotes the law of the solution of the SDE \eqref{fingraph-SDE} with initial law $\nu_0$, and we note that this SDE is unique in law by Assumption (\ref{assumption:B}.4).
Additionally,  if $\mu_0(K^V) =1$ for some compact set $K \subset \R^d$, then (recalling that membership in $\G_2(\cdot)$ requires absolute continuity) \eqref{claim} can be rewritten  as
\begin{align*}
\G_2(\mu_0) = \Big\{Q \circ (X_V(0))^{-1} : Q \in \P(\C^V), \, Q_t \in \G_2(P^{\mu_0}_t) \ \forall t \ge 0\Big\}.
\end{align*}

We conclude this section with the following
  simple corollary of Theorems \ref{th:conditionalindependence-infinitegraph} and \ref{th:gibbsuniqueness}.
  \begin{corollary}
    \label{cor:gibbs}
    Suppose $(G, b, \sigma, \mu_0)$ satisfy Assumption \ref{assumption:B} and $P^{\mu_0}$ represents the unique
    law of the SDE \eqref{fingraph-SDE}.  If $\G_2(\mu_0)$ is a singleton, then  the set $\MRset(\mu_0)$ defined in Theorem \ref{th:gibbsuniqueness} is equal to the singleton $\{P^{\mu_0}\}$,  
    and hence, $P^{\mu_0}$ is completely characterized by its specifications $P_t^{\mu_0}[A| \partial^2 A], t \geq 0,$ for finite $A \subset V$. 
  \end{corollary}
  \begin{proof}
Since Assumption (\ref{assumption:B}.1) ensures that $\mu_0$ is a \mrft \,  with finite second moment,
the set on the left-hand side of \eqref{claim} always contains $\mu_0$.  Thus, if $\G_2(\mu_0)$ is a singleton,
then by Theorem \ref{th:gibbsuniqueness} the
set $\MRset(\mu_0)$
is also a singleton.  On the other hand, by Theorem \ref{th:conditionalindependence-infinitegraph}, for each $t \ge 0$ it holds that
$P_t^{\mu_0}$ is a \mrft \ and thus $P_t^{\mu_0} \in \G_2(P^{\mu_0}_t)$.  Hence, $P^{\mu_0} \in \MRset(\mu_0)$. 
\end{proof}

 \section{Interacting diffusions on a finite graph}
 \label{sec-finitegraph}

In Section \ref{subs-clique} (specifically Proposition \ref{th:hammersleyclifford2}) 
 a useful characterization of a (positive) \mrft \  is derived in an abstract setting.  
This is then used in Section \ref{ap-finite} to prove Theorem  \ref{pr:conditionalindependence-finitegraph};
along  the way   Proposition \ref{prop-unique-finite} is also established. 
 In Sections \ref{se:example-firstorder} and \ref{se:example-marginals} 
 this characterization is applied to  demonstrate via  explicit  examples
 that the space-time \mrft \  property established in Theorem \ref{pr:conditionalindependence-finitegraph} 
 (and hence, Theorem \ref{th:conditionalindependence-infinitegraph})  
cannot in general be substantially improved.

\subsection{Clique factorizations}
\label{subs-clique}

We start by studying the relationship between random fields and factorization properties of their joint density with respect to a given reference measure.
Throughout this section, we work with a fixed finite graph $G=(V,E)$,  as well as a fixed Polish space $\X$, the state space.  
Recall the definition of the diameter  $\mathrm{diam}(A)$ of a set $A \subset V$,
$1$-cliques and $2$-cliques of a graph, and 1st-order and 2nd-order MRFs given
in Section \ref{subs-not}.

First, we recall a well-known theorem often attributed to Hammersley-Clifford, which can be found in various forms in \cite[Theorem 2.30]{georgii2011gibbs} and \cite[Proposition 3.8 and Theorem 3.9]{lauritzen1996graphical}, the latter covering our precise setting.

\begin{proposition}[Hammersley-Clifford] \label{th:hammersleyclifford}
Assume the graph $G=(V,E)$ is finite.  
Assume  $\gmu \in \P(\X^V)$ is absolutely continuous with respect to a product measure $\gmu^* = \prod_{v \in V} \glambda_v \in \P(\X^V)$ for
some $\glambda_v \in \P(\X)$, $v \in V$. 
Consider the following statements: 
\begin{enumerate}[(1)]
\item $\gmu$ is a \mrfo. 
\item The density of $\gmu$ with respect to $\gmu^*$ factorizes in the form
\[
\frac{d\gmu}{d\gmu^*}(x) = \prod_{K \in \rm{cl}_1(G)}f_K(x_K), \qquad x \in \X^V,
\]
for some measurable functions $f_K : \X^K \rightarrow \R_+$, for $K \in \rm{cl}_1(G)$.
\end{enumerate}
Then (2) implies (1). If also $d\gmu/d\gmu^*$ is strictly positive, then (1) implies (2).
\end{proposition}

We next formulate an analogue for a \mrft. 

\begin{proposition}[Second-order Hammersley-Clifford] \label{th:hammersleyclifford2}
Assume the graph $G=(V,E)$ is finite.
Assume $\gmu \in \P(\X^V)$ is absolutely continuous with respect to a product measure $\gmu^* = \prod_{v \in V}\glambda_v \in \P(\X^V)$ for
some $\glambda_v \in \P(\X)$, $v \in V$.  
Consider the following statements:
\begin{enumerate}[(1)]
\item $\gmu$ is a \mrft. 
\item The density of $\mu$ with respect to $\gmu^*$ factorizes in the form
  \begin{equation}
    \label{eq-mufactor}
\frac{d\gmu}{d\gmu^*}(x) = \prod_{K \in \rm{cl}_2(G)}f_K(x_K), \qquad x \in \X^V,
\end{equation}
for some measurable functions $f_K : \X^K \rightarrow \R_+$, for $K \in \rm{cl}_2(G)$.
\end{enumerate}
Then (2) implies (1). If also $d\gmu/d\gmu^*$ is strictly positive, then (1) implies (2).
\end{proposition}
\begin{proof}
Define the \emph{square graph} $G^2 = (V,E')$ by connecting any two vertices of distance $2$. That is, let
\[
E' := \{(u,v) \in V^2 : 1 \le d(u,v) \le 2\},
\]
where $d$ is the graph distance on $G$.
It is straightforward to check the following properties:
\begin{enumerate}[(i)]
\item The $1$-cliques of $G^2$ are precisely the $2$-cliques of $G$. That is, $\mathrm{cl}_2(G)=\mathrm{cl}_1(G^2)$.
\item We have $\partial_{G^2}A = \partial^2_G A$ for any set $A \subset V$.
\end{enumerate}
It follows from (ii) that the statement (1) is equivalent to
\begin{enumerate}
\item[(1')] $\nu$ is a \mrfo \ relative to the graph $G^2$.
\end{enumerate}
On the other hand, it follows from (i) that (2) is equivalent to
\begin{enumerate}
\item[(2')] The density of $\gmu$ with respect to $\gmu^*$ factorizes in the form
\[
\frac{d\gmu}{d\gmu^*}(x) = \prod_{K \in \rm{cl}_1(G^2)}f_K(x_K), \qquad x \in \X^V,
\]
for some measurable functions $f_K : \X^K \rightarrow \R_+$, $K \in \rm{cl}_1(G^2)$. 
\end{enumerate}
The equivalence of (1') and (2') follows from Proposition \ref{th:hammersleyclifford}.
\end{proof}

The \mrft \  property is the more intuitive, but the second property of Proposition \ref{th:hammersleyclifford2} will be quite useful in the analysis as well. Hence, we give it a name.

\begin{definition} \label{def:2cliquefactorization}
  We say that $\gmu \in \P(\X^V)$ \emph{has a $2$-clique factorization with respect to $\gmu^*$} if the density $d\gmu/d\gmu^*$ can be written in the form \eqref{eq-mufactor}. 
\end{definition}

\begin{remark}   
  \label{rem-cutset} 
  For a finite graph $G=(V,E)$ and Polish space $\X$, the following \emph{cutset} characterization of \mrfo's on $\X^V$ is well known: An $\X^V$-valued  random element $(\gX_v)_{v \in V}$ is a \mrfo \   if and only if $\gX_A$ is conditionally independent of $\gX_B$ given $\gX_S$ for any disjoint sets $A,B,S \subset V$ with the property that every path starting in $A$ and ending in $B$ contains at least one vertex of $S$. 
Given  the correspondence between a \mrft \ on a graph and a \mrfo \ on the square graph
(established in the proof of Proposition \ref{th:hammersleyclifford2}), this is easily seen to  imply the following \emph{cutset} characterization of \mrft s: An $\X^V$-valued  random element $(\gX_v)_{v \in V}$ is a \mrft \   if and only if $\gX_A$ is conditionally independent of $\gX_B$ given $\gX_S$ for any disjoint sets $A,B,S \subset V$ with the property that every path starting in $A$ and ending in $B$ contains at least two adjacent vertices of $S$. 
\end{remark}

\subsection{Proof of the second-order Markov random field property for a finite graph}
 \label{ap-finite}

 We now present the proofs of Proposition \ref{prop-unique-finite}  and Theorem  \ref{pr:conditionalindependence-finitegraph}.
Throughout this section, we work with a fixed finite graph $G=(V,E)$ and 
consider the canonical measurable space $\C^V = (\C^V,\text{Borel})$, and 
let $(X_v)_{v \in V} : \C^V \rightarrow \C^V$ denote the canonical processes, that is, 
$X_v((x_u)_{u \in V}) = x_v$ for $x = (x_u)_{u \in V} \in \C^V$, for $v \in V$. 
Let $\mu_0, \mu_0^* \in \P((\R_d)^V)$ be as in Assumption (\ref{assumption:A}.1), and let
  $P^* = P^{*, \mu_0^*} \in \P(\C^V)$ 
 denote the law of the unique solution of the driftless SDE system \eqref{eq:canonical_law} starting from initial law $\mu_0^*$ (the well-posedness of which is given by Assumption (\ref{assumption:A}.3b)).  Recall that $\mu_0^*$ and thus $P^*$ are both product measures. 
 Then, recalling that $dX_v(t) = \sigma_v(t,X_v)\,dW_v(t)$ for $v \in V$, define the following
local martingale (under $P^*$):  
\begin{equation} 
  \label{martv}
M_v(t) := \int_0^t (\sigma_v\sigma_v^\top)^{-1}b_v(s,X_v,X_{N_v(G)}) \cdot dX_v(s), \qquad t \geq 0, 
\end{equation}
 where we use the shorthand notation $(\sigma_v\sigma_v^\top)^{-1}b_v(s, x_v, x_{N_v(G)})$ to denote the map
 \begin{equation}
   \label{sigmamap}
   \R_+ \times \C^V \ni (s,x) \mapsto (\sigma_v(s,x_v)\sigma_v^\top(s,x_v))^{-1}b_v(s, x_v, x_{N_v(G)}) \in \R^d.
   \end{equation}
Also,  given any continuous local martingale $M$, we let 
$\EE(M)$ denote the Doleans exponential: 
\begin{align}
\EE_t(M) = \exp\left( M (t)  - \frac12 [M] (t) \right), \qquad  t \geq 0,  \label{def:doleans-exponential}
\end{align}
where $[M]$ denotes the (optional) quadratic variation process of $M$.

Let  $(f_K)_{K \in \mathrm{cl}_2 (G)}$ be  as in Assumption (\ref{assumption:A}.1).
For each $t > 0$, define the measure $P_t \in \P(\C_t^V)$ by
\begin{align}
  \frac{dP_t}{dP^*_t} & := \dfrac{d\mu_0}{d\mu_0^*} (X_V(0)) \EE_t\left(\sum_{v \in V}M_v\right), \nonumber \\
 	&= \prod_{K \in \mathrm{cl}_2(G)}f_K(X_K(0))\,\prod_{v \in V}\EE_t(M_v), \label{P-factorization}
\end{align}
with $\EE(M)$ and  $P^*$ as defined in the previous paragraph.
Note that $W_v := \int_0^\cdot \sigma^{-1}_v(s,X_v)\,dX_v(s)$,  $v \in V,$ are independent $d$-dimensional Brownian motions  under $P^*$ by Remark \ref{rem-driftless}.
Therefore the stochastic exponentials appearing in \eqref{P-factorization} are true $P^*$-martingales due to the form
of $M_v$ in \eqref{martv}, the linear growth assumption (\ref{assumption:A}.2) on the drifts and the non-degeneracy of $\sigma_v$; 
see Lemma \ref{lem:Girsanov-justification} with $\QQ= P^{*}$, $X=(X_v)_{v \in V}$ and $f_v(t,x) = \sigma_v^{-1}b_v(t,x_v,x_{N_v(G)})$, $v \in V$.   
Further, observe that  $(M_v)_{v \in V}$ are orthogonal under  $P^*$.  
So Girsanov's theorem \cite[Corollary 3.5.2]{karatzas-shreve} implies that under $P_t$,
$\tilde{W}_v := W_v - \int_0^\cdot \sigma_v^{-1}(s, X_v)b_v (s, X_v, X_{N_v(G)})\,ds$, $v \in V$,  are independent $d$-dimensional standard Brownian motions  on $[0,t]$.  From this it follows that under $P_t$, 
  $X$ solves the SDE \eqref{fingraph-SDE} on $[0,t]$, and the same argument also shows that
the restriction to $[0,t]$ of any solution to \eqref{fingraph-SDE} must have law $P_t$ on $\C_t^V$.
Thus, we have uniqueness in law.  
Weak existence follows from Kolmogorov's extension theorem  on observing that $\{P_t, t \geq 0\}$ form a consistent family in the sense that  $P_s$ is the restriction of $P_t$ to $\C_s^V$
for each $t > s > 0$ (due to the martingale property of  $\frac{dP_t}{dP^*_t}$). 
  This completes the proof of Proposition \ref{prop-unique-finite}.

  On the other hand, the fact that for each $t > 0$, $(X_v[t])_{v \in V}$ is a 2MRF on $C_t^V$ 
  follows from \eqref{P-factorization} on applying  
  Proposition \ref{th:hammersleyclifford2} with ${\mathcal X} = \C_t$ and $\mu^* = P^*_t$, noting that $P^*_t$ is a product measure on $\C_t^V$ and that, for each $v \in V$,  $\{v\} \cup N_v(G)$ is a $2$-clique and $M_v$ of \eqref{martv} is $X_{\{v\} \cup N_v(G)}$-measurable.
This proves the first assertion of Theorem \ref{pr:conditionalindependence-finitegraph}.
For the second assertion of Theorem \ref{pr:conditionalindependence-finitegraph}, denote by $P = P^{\mu_0} \in \P(\C^V)$ the law of the unique solution of the SDE system \eqref{fingraph-SDE} with initial law $\mu_0$. 
Fix a finite set $A \subset V$ and bounded continuous functions $f,g,h$ on $\C^A,\C^{\partial^2 A}, \C^{V \setminus (A \cup \partial^2 A)}$, respectively.  
Fix $t > 0$ and let $\Gmc_t := \sigma\{ X_{\partial^2A}[t] \}$ and $\Gmc_\infty := \sigma\{ X_{\partial^2A} \}$.
Below, with some abuse of notation,  for any $B \subset V$, we will also interpret elements  $y \in \C_t^B$  as elements of  $\C^B$  by simply 
  setting $y(s) = y(t)$ for $s \geq t$. 
   Note that with this identification, for any $x \in \C^V$  and $B \subset V$, 
  $x_B[t] \rightarrow x_B$ in $\C^B$ as $t \rightarrow \infty$.
   Then,  noting that $\sigma(\cup_{t > 0}\G_t) = \G_\infty$, 
   invoking the  martingale convergence theorem (in the third equality below),  and using the fact that $P_t=P \circ (X_V[t])^{-1}$ is a \mrft\ on $\C_t^V$ for each $t$ (in the second equality below),  we have
\begin{align}
	& \Emb^P \left[ f(X_A) g(X_{\partial^2A}) h(X_{V \setminus (A\cup\partial^2A)}) \right] \notag \\
	& = \lim_{s \to \infty} \lim_{t \to \infty} \Emb^P \left[ f(X_A[s]) g(X_{\partial^2A}[t]) h(X_{V \setminus (A\cup\partial^2A)}[s]) \right] \notag \\
	& = \lim_{s \to \infty} \lim_{t \to \infty} \Emb^P \left[ \Emb^{P} \left[ f(X_A[s]) \mid \Gmc_t \right] g(X_{\partial^2A}[t]) \Emb^{P} \left[ h(X_{V \setminus (A\cup\partial^2A)}[s]) \mid \Gmc_t \right] \right] \notag \\
	& = \lim_{s \to \infty} \Emb^P \left[ \Emb^P \left[ f(X_A[s]) \mid \Gmc_\infty \right] g(X_{\partial^2A}) \Emb^P \left[ h(X_{V \setminus (A\cup\partial^2A)}[s]) \mid \Gmc_\infty \right] \right] \notag \\
	& = \Emb^P \left[ \Emb^P \left[ f(X_A) \mid \Gmc_\infty \right] g(X_{\partial^2A}) \Emb^P \left[ h(X_{V \setminus (A\cup\partial^2A)}) \mid \Gmc_\infty \right] \right], \label{eq:infinitetime}
\end{align}
where we have also made  repeated use of the boundedness and continuity of $f, g, h$ and the bounded convergence theorem. 
This shows that $X_A$ and $X_{V \setminus (A \cup \partial^2A)}$ are conditionally independent given $X_{\partial^2A}$ under $P$, that is, $P$ is a \mrft \ on $\C^V$.
This completes the proof.
\hfill \qedsymbol

\subsection{Illustrative examples}
\label{subs-ceg}

We now provide examples to show that the \mrft \  property cannot in general be strengthened. 

\subsubsection{The failure of the first-order MRF property for trajectories} \label{se:example-firstorder}

In general, $P_t$  fails to be a first-order Markov random field on $\C^V_t$ for any $t > 0$, even if the initial states are i.i.d. To see why, notice that the density $dP_t/dP^*_t$ given by \eqref{P-factorization} does not in general admit a clique factorization. Indeed, for $v \in V$ and $t > 0$, we recall the definition of $M_v$ from \eqref{martv} and $\EE_t(M_v)$ from
\eqref{def:doleans-exponential}, which we write in full as
\begin{align*}
\EE_t(M_v) = \exp\Bigg(&\int_0^t(\sigma_v\sigma_v^\top)^{-1}b_v(s,X_v,X_{N_v(G)})\cdot dX_v(s)  \\
	&- \frac12\int_0^t\left\langle b_v(s,X_v,X_{N_v(G)}),\,(\sigma_v\sigma_v^\top)^{-1}b_v(s,X_v,X_{N_v(G)}) \right\rangle ds\Bigg).
\end{align*}
Noting that $\{v\} \cup N_v(G)$ is a $2$-clique but not a $1$-clique, this reveals why one cannot hope for a factorization over $1$-cliques. For example, consider the ``nice" case where $\sigma_v \equiv I$ and $b_v(s,x_v,x_{N_v(G)}) = \sum_{u \in N_v(G)}(x_u(s)-x_v(s))$.
(Equivalently, $b_v(s,x_v,x_{N_v(G)}) = \nabla_{x_v}h(x)$ is of gradient-type with potential $h(x) = -\frac12\sum_{(u,v) \in E}|x_u-x_v|^2$, where $E$ is the edge set of $G$.)  
Then the first term in the above exponential splits nicely into a sum of pairwise interactions $\sum_{u \in N_v(G)}\int_0^t(X_u(s)-X_v(s)) \cdot dX_v(s)$, but the second term becomes
\[
-\frac12\sum_{u,w \in N_v(G)}\int_0^t\lan X_u(s)-X_v(s), \, X_w(s)-X_v(s)\ran \,ds.
\]
It is this term which fails to factorize further over $1$-cliques as opposed to $2$-cliques and thus precludes the first-order Markov property whenever $d\mu_0/d\mu_0^*$ is strictly positive due to Proposition \ref{th:hammersleyclifford}.

To informally provide a different  (but arguably more intuitive)  perspective on why the first-order Markov property for past histories fails,
consider the case when $G$ is a line segment of
length $\ell = 3$, labelling the vertices $-1, 0, 1$.
Then, although the driving Brownian motions are all independent and
the dynamics of each of the two extreme vertices only depend on its own state and the state of the center vertex, at any
time $t$, conditioning on the past history of the states of the center vertex, does not make  $X_{-1}(t)$
independent of $X_1(t)$ because the conditioning correlates the Brownian motions $W_{-1}$ and $W_1$ on the interval
$[0,t]$.  This happens because the past history of $X_0$ is influenced by both $W_{-1}$ and $W_1$ via $X_{-1}$ and $X_1$.
On the other hand, to see why the \mrft \ property nevertheless does hold,
note that if $G$ were a line segment of length $4$, labeling the vertices
$\{-2,-1,1,2\}$, then conditioning on the history of the states of the two center vertices $-1$ and $1$ no longer correlates the 
Brownian motions $W_{-2}$ and $W_2$ 
since the dynamics of each of  the conditioned vertices depends  on a different
driving Brownian motion.   Thus, although the conditioning changes the distribution of $W_{-2}$ and $W_2$ (for instance, they need no
longer be Brownian motions), 
they remain independent, and hence $X_{-2}(t)$ is conditionally independent of $X_2(t)$ in this case.

\begin{remark}
There are certain situations in which $P_t$ is, in fact, a \mrfo for each $t > 0$ (even though we know from the above examples that this is not in general the case). For example, suppose that for every $v \in V$, 
there exists a clique $K_v$ of $G$ with $v \in K_v \subset N_v(G)$ such that $b_v(t,x_v,x_{N_v(G)})=\widetilde{b}_v(t,x_v,x_{K_v})$ depends on $x_{N_v(G)}$ only through $x_{K_v}$. Suppose also that $d\mu_0/d\mu_0^*$ admits a $1$-clique factorization. Then, recalling \eqref{P-factorization}, note that for each $v$ the martingale $M_v$ is measurable with respect to $X_{K_v}$, and deduce from Proposition \ref{th:hammersleyclifford} that $P_t$ is a first-order Markov random field.
For a concrete example that has the above form,
consider the case when 
$G$ is a triangular lattice with $V = \{\o, 0, 1, \ldots, m\}$, for some $m \in \N$,  
with the central vertex ${\o}$ having the neighborhood $N_{\o}(G) = \{0, \ldots, m\}$ and 
for each $v \in V \setminus \{\o\}$,  $N_v(G) := \{{\o}, v+1, v-1\}$, where the vertices are to be interpreted
mod $m+1$.  
Further, suppose the initial conditions are i.i.d.  and  that for some $c \in \R$,
$b_v (t, x_v, x_{N_v(G)}) = c(x_{{\o}} + x_{v+1})$ for $v \in V \setminus \{\o\}$ and $b_{\o} (t, x_{\o}, x_{N_{\o}(G)}) = cx_{{\o}}$. 
Then  this provides a specific example with $K_v = \{\o,v+1\} \subset N_v(G)$ for $v \in V \setminus \{\o\}$ and $K_{\o}=\emptyset$.
In a similar spirit,  the directed cycle graph  model of \cite{DetFouIch18} provides another example.
\end{remark}

\subsubsection{The failure of MRF properties for time-$t$ marginals} \label{se:example-marginals}

It is natural to wonder if and when the time-$t$ marginals $P_t \circ X(t)^{-1} \in \P((\R^d)^V)$ remain a first- or second-order
Markov random field, given that this property is true at time $0$, or even given i.i.d.\ initial conditions.
 This question is related to propagation of Gibbsianness and 
 Gibbs-non-Gibbs transitions that have been studied in the literature,  which is discussed in greater detail in
Remark \ref{rem-gibbsprop}.

  Here,  we provide a simple example where  both the first-order and  second-order Markov property  fail for time-$t$ marginals. In fact, in this simple model we will see that there is no non-trivial conditional independence structure. 
Consider the segment with $5$ vertices:  $G=(V,E)$ given by $V=\{1,2,3,4,5\}$ and $E = \{(i,i+1) : i=1,2,3,4\}$, and consider the SDE system
\begin{align}
\begin{split}
	dX_1(t) & = (X_2(t)-2X_1(t))\,dt+dW_1(t), \\
	dX_i(t) & = (X_{i-1}(t)+X_{i+1}(t) - 2X_i(t))\,dt+dW_i(t), \qquad i=2,3,4, \\
	dX_5(t) & = (X_4(t)-2X_5(t))\,dt+dW_5(t),
\end{split}
\label{def:ex:SDE}
\end{align}
with $X_i(0)=0$ for each $i$. Once again, note that the drift here is of gradient type with potential $h(x) =\sum_{i=1}^4 x_i x_{i+1}-\sum_{i=1}^5 x_i^2$. 
Letting $\bm{X}(t)$ denote the column vector $(X_1(t),\ldots,X_5(t))$ and similarly for $\bm{W}(t)$, we may write this in vector form as
\begin{align}
d\bm{X}(t) = L\bm{X}(t)\,dt + d\bm{W}(t), \label{def:multi-OU}
\end{align}
where $L = A - 2I$ is the adjacency matrix $A$ of the graph minus twice the identity $I$: 
\begin{equation*}
	L = \begin{pmatrix}
	-2 & 1 & 0 & 0 & 0 \\
	1 & -2 & 1 & 0 & 0 \\
	0 & 1 & -2 & 1 & 0 \\
	0 & 0 & 1 & -2 & 1 \\
	0 & 0 & 0 & 1 & -2 
	\end{pmatrix}.
\end{equation*}
The solution of the SDE \eqref{def:multi-OU} is given by
\begin{align*}
\bm{X}(t) = e^{Lt} \int_0^t e^{-Ls} \,d\bm{W}(s), \qquad t > 0.
\end{align*}
Noting that $L$ is symmetric and invertible, we deduce that $\bm{X}(t)$ is jointly Gaussian with mean zero and covariance matrix 
\begin{align}
\E[\bm{X}(t)\bm{X}(t)^\top] = \int_0^t e^{2Ls} \,ds = \frac12 L^{-1}(e^{2Lt}-I). \label{def:multi-OU-cov}
\end{align}
This covariance matrix can easily be computed explicitly by noting that the tridiagonal Toeplitz matrix $A$ is explicitly diagonalizable. To spare the reader any tedium, we provide only some pertinent snapshots. At time $t=2$ the covariance matrix is
\begin{align}
\E[\bm{X}(2)\bm{X}(2)^\top] = \begin{pmatrix}
0.3611 & 0.2388 & 0.1435 & 0.0767 & 0.0324 \\
0.2388 & 0.5046 & 0.3156 & 0.1759 & 0.0767 \\
0.1435 & 0.3156 & 0.5370 & 0.3156 & 0.1435 \\
0.0767 & 0.1759 & 0.3156 & 0.5046 & 0.2388 \\
0.0324 & 0.0767 & 0.1435 & 0.2388 & 0.3611
	\end{pmatrix}. \label{def:ex-covmatrix}
\end{align}
Using the well known formula for conditional measures of joint Gaussians, we compute from this that
\begin{align*}
\mathrm{Cov}(X_1(t),X_3(t)|X_2(t))=
\begin{pmatrix}
0.2481 & -0.0058 \\
-0.0058 & 0.3397
\end{pmatrix},
\end{align*}
which reveals that $X_1(t)$ and $X_3(t)$ are not conditionally independent given $X_2(t)$. Hence, $(X_i(t))_{i \in G}$ is not a first-order Markov random field.
Similarly, by computing
\begin{align*}
\mathrm{Cov}(X_1(t),X_4(t)|X_2(t),X_3(t))=
\begin{pmatrix}
0.2480 & -0.0030 \\
-0.0030 & 0.3189
\end{pmatrix},
\end{align*}
we see that $X_1(t)$ and $X_4(t)$ are not conditionally independent given $(X_2(t),X_3(t))$. Hence, $(X_i(t))_{i \in G}$ is not a second-order Markov random field.

In fact, in this example, there is no non-trivial conditional independence structure, in the sense that there are no two vertices $i,j$ such that $X_i(t)$ and $X_j(t)$ are conditionally independent given $\{X_k(t) : k \in G\setminus\{i,j\}\}$ for some $t > 0$.
This can be read off from the the so-called \emph{precision matrix}, which is simply the inverse of the covariance matrix, $Q(t) := (\E[\bm{X}(t)\bm{X}(t)^{\top}])^{-1}$.
As is well known and can easily be seen from the form of the multivariate Gaussian density, the precision matrix reveals the conditional independence structure (see, e.g., \cite[Proposition 5.2]{lauritzen1996graphical}), in the following sense: For $t > 0$ define the graph $\widetilde{G}(t) = (V,\widetilde{E}(t))$ with the same vertex set $V$ but with $(i,j) \in \widetilde{E}(t)$ if and only if $Q_{i,j}(t)  \neq 0$. Then $\bm{X}(t)$ is a (first-order) Markov random field with respect to the graph $\widetilde{G}(t)$. In our example, $\widetilde{G}(t)$ is the complete graph for each $t > 0$, and this Markov property is vacuous. (Note, however, that $Q(t) \to 2L$ as $t\to\infty$ because $L$ is negative definite, and the unique invariant measure of this diffusion is a first-order Markov random field with respect to the original graph $G$.)

A variation on this example gives rise to another interesting phenomenon. Suppose we modify the example by replacing the diagonal entries of $L$ with zeros, i.e., remove all the $-2X$ terms from the drifts in \eqref{def:ex:SDE}. Then the covariance matrix is again invertible, and now $Q_{1,4}(t)=Q_{2,5}(t)=0$ for all $t > 0$, where we continue with the notation of the previous paragraph. That is, $\widetilde{G}(t)$ is not the complete graph, but rather the complete graph with the edges $(1,4)$ and $(2,5)$ removed, for each $t > 0$. In particular, $X_1(t)$ and $X_4(t)$ are conditionally independent given $(X_2(t),X_3(t),X_5(t))$, for each $t > 0$.

\begin{remark}
  \label{rem-gibbsprop}  
    As mentioned in the introduction, one motivation for studying such conditional independence questions is that (a stronger version of) the   
  MRF structure of interacting SDEs  can lead to a more succinct autonomous ``local characterization" of 
  the dynamics at a vertex and its neighborhood, as developed in the quite different setting of unimodular
  Galton-Watson trees in  \cite{LacRamWu-original,LacRamWu20diffusion}.   From this perspective,
  it would be of interest to investigate if there are  non-trivial special cases   when
  the  first-order or second-order MRF property for time-$t$ marginals propagates.
  A different but related question that has been studied in the literature is  propagation of Gibbsianness for 
  an infinite system of interacting real-valued  diffusions indexed by $\Z^d$. 
  Specifically, the work  \cite{DerRoe05} considers a collection of interacting diffusions, 
  indexed by $\Z^d$, with
  identity covariance and  a  drift that  is the gradient of a Hamiltonian function
  associated with a certain interaction
  potential $\Phi$, and with an initial distribution that is also a
  Gibbs measure (as in Section \ref{se:gibbs})
  with respect to a Gibbsian specification (in the sense of \cite[Chapter 2]{georgii2011gibbs})
  associated with another  interaction potential $\Phi_0$, where
  both interaction potentials $\Phi$ and $\Phi_0$   are  assumed to be of finite range and 
  satisfy certain smoothness conditions. 
  It is shown in \cite{DerRoe05} that when  either $t$ or the interaction strength is sufficiently small,
  the time-$t$ marginals are strongly Gibbsian, that is, associated with Gibbsian specifications that have 
  an absolutely summable, though not necessarily  finite range, interaction potential. 
  Extensions of these results to the case of interacting real-valued diffusions on $\Z^d$ 
  with non-Markovian drifts with finite memory (again with finite range interactions and identity covariance)
  were later  obtained in \cite{RedRoeRus10} and \cite{RoeRus14}.
  The restrictions on the time and interaction strength in these works arise from the fact that
  perturbative arguments are used. 
  However, in general for moderate interaction strengths and moderate times, the time-$t$ marginals  
can fail to be Gibbsian   (see, e.g., \cite{van2009gibbsianness}, as well as the survey \cite{van2010gibbs}, which also discusses related results for spin systems). 
\end{remark}

\section{Finite-graph approximations for Markov random fields} \label{se:2MRFs}

In this section we establish some important preparatory results that are used in the proof of Theorem 
\ref{th:conditionalindependence-infinitegraph}, which extends the finite graph results of
Theorem \ref{pr:conditionalindependence-finitegraph} to the infinite graph setting. 
Fix $(G, b, \sigma, \mu_0)$ that satisfy Assumption \ref{assumption:B} and suppose $G = (V,E)$ is  countably infinite.
Recall  that
$P =  P^{\mu_0} \in \P(\C^V)$ and $P^{*,\mu_0} \in \P(\C^V)$ 
denote the unique law of the SDE systems \eqref{fingraph-SDE} and \eqref{eq:canonical_law}, respectively, 
both with initial laws $\mu_0$, which are well-posed by Assumptions (\ref{assumption:B}.4) and  (\ref{assumption:B}.3). 
To show that    $P_t = P^{\mu_0}_t$  forms a \mrft \  on  $\C_t^V$, we can no longer
apply the clique factorization arguments used for finite graphs because the formula \eqref{P-factorization} does
not extend to infinite graphs.  Even worse, the density $dP_t/dP^*_t$ therein does not exist, and it seems impossible to establish directly
that by projecting to a finite set $A \subset V$ we have a density $dP_t[A]/dP^*_t[A]$ that admits
a 2-clique factorization.  
Instead, we  approximate the measure on the infinite graph by \mrft s on a  growing sequence of finite graphs,
arguing that the desired \mrft \  property  passes to the limit.
To highlight some of the subtleties that arise in such an approximation argument,
and to better motivate the other results established in this section,  
we first desribe the approximating sequence of measures  in Section \ref{subs-approxmeas}.
Then in the subsequent two sections we establish some general properties of finite-graph \mrft s to be
 used in the proof of Theorem \ref{th:conditionalindependence-infinitegraph} in Section \ref{subs-pf-infingraph}, 
which  are also of independent interest.

\subsection{Construction of the approximating sequence of SDEs}
\label{subs-approxmeas}

We fix $(G,b,\sigma, \mu_0)$ that satisfy Assumption \ref{assumption:B}.  
As in Section \ref{ap-finite}, we will
work with the canonical measure space $\C^V = (\C^V,\text{Borel}, P^{*,\mu_0})$, and 
let $(X_v)_{v \in V} : \C^V \rightarrow \C^V$ again denote the canonical processes.
Also, recall from Section \ref{subs-not} that given any measurable
space $\X$,  measure $\mu \in \P( \X^V)$ and subset  $U \subset V$,
$\mu[U] \in \P(\X^U)$ denotes the restriction of $\mu$ to the set $\X^U$.

Let $V_n \subset V, n \in \N,$ be such that $\bigcup_{n} V_n = V$, and let  $G_n = (V_n, E_n)$, for some  edge set $E_n$ to be specified later.
Also, for each $n \in \N$ and $v \in V_n$, let 
 $b^n_v: \R_+ \times \C \times \C^{N_v(G_n)} \mapsto \R^d$ be any progressively measurable map
that    satisfies the same conditions as $b_v$ in Assumption (\ref{assumption:B}.2).    
Fix $t > 0$ and for each $n$, define $P_t^n = P_t^{\mu_0,n} \in \P(\C_t^V)$ by
\begin{align}
  \label{dPdPstar}
\frac{dP^n_t}{dP^{*,\mu_0}_t} = \frac{dP_t^{\mu_0,n}}{dP^{*,\mu_0}_t}= \prod_{v \in V_n}\EE_t\left(\int_0^\cdot (\sigma_v\sigma_v^\top)^{-1}b^n_v(s,X_v,X_{N_v(G_n)}) \cdot dX_v(s)\right), 
\end{align}
where, as before,  $(\sigma_v \sigma_v^\top)^{-1} b^n_v(s,x,x_{N_v(G_n)})$ denotes the map \eqref{sigmamap}. 
We can apply Lemma \ref{lem:Girsanov-justification} with $\QQ= P^{*,\mu_0}$, $X=(X_v)_{v \in V_n}$
and $f(t,x) = (\sigma_v^{-1}b^n_v(t,x_v,x_{N_v(G_n)}))_{v \in V_n}$,  to conclude that 
the stochastic exponential  in \eqref{dPdPstar} is a true  $P^{*,\mu_0}$-martingale, due to the linear
growth, non-degeneracy and  boundedness properties of $b^n_v$ and $\sigma_v$ in Assumptions (\ref{assumption:B}.2) and
(\ref{assumption:B}.3).   
Hence, the family $(P^n_t)_{t > 0}$ is consistent in the sense that the restriction of $P^n_t$ to $\C_s^V$ is precisely $P^n_s$ for each $t > s > 0$. Thus, by the Kolmogorov extension theorem, $(P^n_t)_{t > 0}$ uniquely determines a probability measure $P^n$ on $\C^V$.
Now, from \eqref{dPdPstar},  \eqref{eq:canonical_law} and Girsanov's theorem
 \cite[Corollary 3.5.2]{karatzas-shreve}, it follows that under $P^n$ the canonical process solves the SDE system 
\begin{equation}
  \label{fingraph-SDEn} 
\begin{array}{rcl}
dX_v(t) &=& b^n_v(t,X_v,X_{N_v(G_n)})\,dt + \sigma_v(t,X_v)\,dW_v(t), \quad v \in V_n,  \\
dX_v(t) &=& \sigma_v(t,X_v)\,dW_v(t), \quad v \in V \setminus V_n,
\end{array}
\end{equation}
with $(X_v(0))_{v \in V} \sim \mu_0$,  where $(W_v)_{v \in V}$ are independent Brownian motions under $P^n$. 
Note that for $v \in V_n$,  the third argument of $b^n_v$ 
looks only at the states in $N_v(G_n)$, and thus
 $b^n_v$ depends only on the states of vertices in $G_n$.
Thus, $P^n[V_n]$ is precisely the law of the finite-graph SDE system \eqref{fingraph-SDE} with inputs $(G,(b^n_v)_{v \in V},(\sigma_v)_{v \in V},\mu_0[V_n])$.

In order to implement our approximation argument we would like   
to choose $G_n$ and $(b^n_v)_{v \in V}$ such that both $P^n \to P$ and  each $P^n_t[V_n]$ is a \mrft. 
In order to have $P^n \to P$ we should naturally choose $V_n$ increasing to $V$ and $b^n_v$ to behave like $b_v$ for most $v$. 
But the \mrft \ property is more delicate. It would follow from the finite-graph result of Theorem \ref{pr:conditionalindependence-finitegraph} that $P^n_t[V_n]$ is a \mrft \ on $\C^{V_n}$ 
only if $\mu_0[V_n]$ were a \mrft \  on $(\R^d)^{V_n}$.   But  $\mu_0[V_n]$ is not necessarily
a \mrft \  for 
arbitrary  $V_n$ (e.g., with  $G_n$  the induced subgraph),
 even though  $\mu_0$ is a \mrft \ on the full graph $G$ by assumption; 
in other words, the \mrft \ property is not in general preserved under projections, as illustrated in Example \ref{eg-fingraph} below. 
 However, in  Section \ref{se:MRFs-projections}  
we show that for any Markov random field on an infinite
graph $G = (V,E)$, it is possible to identify  a suitable increasing sequence of vertices $(V_n)_{n \in \N}$ and
associated  graph $G_n = (V_n,E_n)$ for each $n \in \N$  that is a slight modification of the induced subgraph on $V_n$, such that
the desired projection property holds.
Then, in Section \ref{se:conditional distributions of 2MRFS} we prove some results on preservation of  a class of
conditional distributions of \mrft s  under restriction to induced subgraphs. 
The above results are 
combined with tightness and convergence estimates for the approximating sequence $\{P^n\}$ obtained 
in Section \ref{subs-tightness} to  complete the proof of Theorem \ref{th:conditionalindependence-infinitegraph}
in  Section \ref{subs-pf-infingraph}.

\subsection{Projections of Markov random fields} \label{se:MRFs-projections}

We first provide a simple example to illustrate that the restriction of an MRF to an induced subgraph need not remain an MRF.

\begin{example}
  \label{eg-fingraph}
  Suppose $G$ is a finite two-dimensional lattice, with vertex set $V$ identified with $\{-n,\ldots,n\}^2$ and the usual nearest-neighbor edge set, and let $(\gX_v)_{v \in G}$ be a \mrfo \  on $\R^V$. 
  Consider the line subgraph  $H=\{(i,0)  : i=-n,\ldots,n\}$ in $G$, and consider the restriction  $(\gX_v)_{v \in H}$ of the \mrfo \   to $H$. 
  Note that every path in $H$ that starts in $A := \{(0,0)\}$ and ends in $B = \{ (2,0)\}$ must
  traverse through the vertex of $S:= \{(1,0)\}$.  Thus, by the cutset characterization of \mrfo's given in Remark \ref{rem-cutset},
  for $(\gX_v)_{v \in H}$ to be an  \mrfo \  on $H$,  $\gX_{(0,0)}$ must be conditionally independent of  $\gX_{(2,0)}$  given $\gX_{(1,0)}$.
  However, by the same cutset characterization,  it is clear that this conditional independence
  cannot be deduced from the \mrfo \  property of $(\gX_v)_{v \in G}$ on $G$ since   there are  paths in $G$ that start in $A$
  and end in $B$ that are disjoint from $S$. 
  Similarly, if we assume $(\gX_v)_{v \in G}$ is a \mrft, the configuration $(\gX_v)_{v \in H}$ can fail to be a \mrft.
\end{example}
    
This example does suggest, however, that we can restore the MRF property by enlarging the edge set of the induced subgraph to reflect the lost connectivity. The following lemma gives one way to do this which is certainly not the only way, but it serves our purpose.
For a random element $(\gX_v)_{v \in V}$ of $\X^V$ with law $\gmu \in \P(\X^V)$, and for a set $A \subset V$, recall that we write 
$\gmu[A]$ to denote the law of $\gX_A$, the coordinates in $A$.

\begin{lemma} \label{le:MRFprojections}
Fix a rooted graph $G=(V,E,\o)$ and $n \ge 4$. Define $V_n := \{v \in V : d(v,\o) \le n\}$ and $U_n := V_n \setminus V_{n-2}$, where $d$ denotes the graph distance.
Define a graph $G_n=(V_n,E_n)$, where
\[
E_n := \{(u,v) \in V_n \times V_n : (u,v) \in E\} \cup \{(u,v) \in U_n \times U_n, u \neq v \}.
\]
\begin{enumerate}[(i)]
\item For any $A \subset V_{n-3}$, it holds that $\partial^2_G A = \partial^2_{G_n} A$.
  Also, for any $A' \subset V_{n-2}$, $\partial^2_G A' \subset \partial^2_{G_n} A'$.
\item  If $K \in \mathrm{cl}_2(G)$ satisfies $K \subset V_n$, then $K \in \mathrm{cl}_2(G_n)$.
\item If a $\X^V$-valued random variable $(\gX_v)_{v \in V}$ is a \mrft \  with respect to  $G$, then $(\gX_v)_{v \in V_n}$ is a \mrft \  with respect to $G_n$.
\item Suppose $V$ is finite and the law $\gmu$ of $\gX^V$ admits the following $2$-clique factorization with respect to a product measure $\gmu^* = \prod_{v \in V}\glambda_v \in \P(\X^V)$ for some $\glambda_v \in \P(\X)$, 
\[
\frac{d\gmu}{d\gmu^*}(x_V) = \prod_{K \in \mathrm{cl}_2(G)}f_K(x_K),
\]
for some measurable functions $f_K : \X^K \rightarrow \R_+$, for $K \in \mathrm{cl}_2(G)$. Then $\gmu[V_n]$ admits a $2$-clique factorization
\[
\frac{d\gmu[V_n]}{d\gmu^*[V_n]}(x_{V_n}) = \prod_{K \in \mathrm{cl}_2(G_n)}f^0_K(x_K),
\]
for some measurable functions $f^0_K : \X^K \rightarrow \R_+$, for $K \in \mathrm{cl}_2(G_n)$, which additionally satisfy the consistency condition $f^0_K \equiv f_K$ for $K \in \mathrm{cl}_2(G)$ such that $K \subset V_{n-3}$. 
\end{enumerate}
\end{lemma}
\begin{proof} {\ }

\begin{enumerate}[(i)]
\item From the definition of $E_n$ it follows quickly that (a) for $A' \subset V_{n-2}$, $\partial_G A' = \partial_{G_n} A'$, and
  (b) for $A'' \subset V_{n-1}$, $\partial_G A'' \subset \partial_{G_n} A''$.
  Iterate these observations to prove the claims. 
\item Let $d_G$ and $d_{G_n}$ denote the graph distance in $G$ and $G_n$, repsectively. From the definition of $E_n$, it is straightforward to argue that $d_{G_n} \le d_G$ on $V_n \times V_n$. Indeed, for any $u,v \in V_n$ and any path from $u$ to $v$ in $G$, there is a path from $u$ to $v$ in $G_n$ which is not longer. 
  This implies for every $u, v \in V_n$, $d_G(u,v) \leq 2$ implies $d_{G_n}(u,v) \leq 2$, which proves property (ii).
 \item Let $(\gX_v)_{v \in V}$ be a \mrft \  with respect to $G$.
   Let $A \subset V_n$, $B = V_n \setminus (A \cup \partial^2_{G_n} A)$, and $S := \partial^2_{G_n} A$. 
   Assuming without loss of generality that $A$ and $B$ are nonempty, we must prove that
   $\gX_A$ and $\gX_B$ are conditionally independent given $\gX_S$.  
First notice that one cannot have both $A \cap U_n \ne \emptyset$ and $B \cap U_n \ne \emptyset$, as this would imply $d_{G_n}(A,B) \le 1$, contradicting the definition of $B$.
Therefore we must have either $A \cap U_n = \emptyset$, $B \cap U_n = \emptyset$, or both.

\textit{Case 1:} Suppose $A \cap U_n = \emptyset$. This means $A \subset V_{n-2}$ and hence $\partial^2_G A \subset S$ by (i). Since $\gX_A$ and $\gX_{V \setminus (A \cup \partial^2_G A)}$ are conditionally independent given $\gX_{\partial^2_G A}$, we then have conditional independence of $\gX_A$ and $\gX_B$ given $\gX_S$. Indeed, this uses the elementary fact that if $(Z_1,Z_2,Z_3,Z_4)$ are random variables with $Z_1$ conditionally independent of $(Z_2,Z_3)$ given $Z_4$, then $Z_1$ is conditionally independent of $Z_2$ given $(Z_3,Z_4)$. 

\textit{Case 2:} Suppose $B \cap U_n = \emptyset$. This means $B \subset V_{n-2}$ and hence, again by (i),  $\partial^2_G B \subset \partial^2_{G_n} B$. Also note that $\partial^2_{G_n} B \subset S$ (since otherwise $A \cap \partial^2_{G_n} B \ne \emptyset$, which contradicts the definition of $B$). Since the \mrft \  property with respect to $G$ implies $\gX_B$ and $\gX_{V \setminus (B \cup \partial^2_G B)}$ are independent conditioned on $\gX_{\partial^2_G B}$, we then have conditional independence of $\gX_B$ and $\gX_A$ given $\gX_S$. Since $A \subset V_n$ was arbitrary,
this  proves that $(\gX_v)_{v \in V_n}$ is a \mrft \  with respect to $G_n$.    
\item Let $\K_n$ denote the set of $K \in \mathrm{cl}_2(G)$ such that $K \subset V_n$. Recalling that $\gmu^*$ is a product measure, using the assumed clique factorization of $\gmu$, we can then write  
\begin{align*}
\frac{d\gmu[V_n]}{d\gmu^*[V_n]}(x_{V_n}) &= \int_{\X^{V \backslash V_n}} \prod_{K \in \mathrm{cl}_2(G)}f_K(x_K) \,\gmu^*[V \backslash V_n](dx_{V \backslash V_n}) \\
	&= \prod_{K \in \K_n}f_K(x_K)\int_{\X^{V \backslash V_n}} \prod_{K \in \mathrm{cl}_2(G) \backslash \K_n}f_K(x_K) \,\gmu^*[V \backslash V_n](dx_{V \backslash V_n}).
\end{align*}
Now note that any $K \in \mathrm{cl}_2(G) \backslash \K_n$ is not contained in $V_n$, and as a $2$-clique it can have no neighbors in $V_{n-2}$. Recalling that $U_n = V_n \backslash V_{n-2}$, we see that the integral expression is $x_{U_n}$-measurable; that is, there is a measurable function $g_n : \X^{U_n} \rightarrow \R_+$ such that
\[
g_n(x_{U_n}) = \int_{\X^{V \backslash V_n}} \prod_{K \in \mathrm{cl}_2(G) \backslash \K_n}f_K(x_K) \,\gmu^*[V \backslash V_n](dx_{V \backslash V_n}).
\]
Note that $U_n \in \mathrm{cl}_2(G_n)$ by definition of $G_n$. Since clearly $\K_n \subset \mathrm{cl}_2(G_n)$, we find that the expression
\begin{align*}
\frac{d\gmu[V_n]}{d\gmu^*[V_n]}(x_{V_n}) &= \prod_{K \in \K_n}f_K(x_K)g_n(x_{U_n})
\end{align*}
exhibits a $2$-clique factorization of $\gmu[V_n]$ over the graph $G_n$ satisfying the desired consistency condition.
\end{enumerate}
\end{proof}

\subsection{Conditional distributions of second-order Markov random fields} \label{se:conditional distributions of 2MRFS}
First, in Lemma \ref{le:specification-abstract}, given a \mrft \  with respect to a graph, and another \mrft \  on a subgraph, or
 more generally given MRFs on two  overlapping graphs, we identify  conditions under which the 
conditional distributions of a subset  in the intersection (given its complement) coincide for both  \mrft s. 
This will be used  to establish, for a suitable choice of $b^n$, a certain consistency condition for
the sequence of approximating measures  $\{P^n\}$  used in the proof of Theorem \ref{th:conditionalindependence-infinitegraph}.
Let us briefly recall a notation we introduced more carefully just before Theorem \ref{th:gibbsuniqueness}: For $\nu \in \P(\X^V)$ and $A,B \subset V$ we write $\nu[A \, | \, B]$ for the conditional law of the $A$-coordinates given the $B$-coordinates.

\begin{lemma} \label{le:specification-abstract}
Let $G=(V_G,E_G)$ and $H=(V_H,E_H)$ be finite graphs, and assume $V^* \subset V_G \cap V_H$ satisfies
\begin{align}
E_G \cap (V^* \times V^*) = E_H \cap (V^* \times V^*). \label{asmp:edgesets}
\end{align}
Moreover, let $A \subset V^*$ satisfy $\partial^2_G A \subset V^*$ and $\partial^2_H A \subset V^*$. Then $\partial^2_H A = \partial^2_G A =: \partial^2A$, and it holds that
\begin{align}
\{K \in \mathrm{cl}_2(G) : K \cap A \neq \emptyset\} = \{K \in \mathrm{cl}_2(H) : K \cap A \neq \emptyset\} =: \mathcal{K}_A. \label{eq:cliques}
\end{align}
Next, let $\gmu^H \in \P(\X^{V_H})$ and $\gmu^G \in \P(\X^{V_G})$, and suppose there exists a product measure
$\gmu^* = \prod_{v \in V_G \cup V_H}\glambda_v \in \P(\X^{V_G \cup V_H})$ for some $\glambda_v \in \P(\X)$,  $v \in V_G \cup V_H$, such that the densities factorize as 
\begin{align*}
\frac{d\gmu^H}{d\gmu^*[V_H]}(x_{V_H}) =  \prod_{K \in \mathrm{cl}_2(H)}f^H_K(x_K), \quad\quad \frac{d\gmu^G}{d\gmu^*[V_G]}(x_{V_G}) =  \prod_{K \in \mathrm{cl}_2(G)}f^G_K(x_K), 
\end{align*} 
for measurable functions $(f_K^H:  \X^K \mapsto \R_+)_{K \in \mathrm{cl}_2(H)}$ and $(f_K^G:  \X^K \mapsto \R_+)_{K \in \mathrm{cl}_2(G)}$ satisfying $f_K^H \equiv f_K^G$ for all $K \in \mathcal{K}_A$.
Then  $\gmu^H[ A\, | \, \partial^2A] = \gmu^G[ A\, | \, \partial^2A]$, almost surely with respect to $\gmu^*[\partial^2 A]$. 
\end{lemma}
\begin{proof}
  Let $A \subset V^*$ satisfy  $\partial^2_G A \subset V^*$ and $\partial^2_H A \subset V^*$. It is immediate from \eqref{asmp:edgesets} that $\partial^2_H A = \partial^2_G A$, and we write simply $\partial^2A$ for this set. To check \eqref{eq:cliques}, note that if $K \in \mathrm{cl}_2(G)$ intersects $A$, then $K \subset A \cup \partial^2A \subset V^*$. By \eqref{asmp:edgesets} the edge sets of $G$ and $H$ agree when restricted to $V^*$, and we deduce that $K \in \mathrm{cl}_2(H)$; this proves $\subset$ in \eqref{eq:cliques}, but the reverse inclusion follows by the same argument.
  Note that, with  $\mathcal{K}_A$ as defined in \eqref{eq:cliques}, we have also shown that
  \begin{equation}
    \label{prop-cliques}
    K \in \mathcal{K}_A \Rightarrow K \subset V^*. 
    \end{equation}

Let us work in the rest of the proof on the canonical probability space $(\X^{V_G \cup V_H},\mathrm{Borel}, \nu^*)$, with $\E$ denoting expectation on this space,  and all equations are understood to hold $\nu^*$-almost surely. 
Let $\id =(\id_v)_{v \in V_G\cup V_H}$ denote the identity map on $\X^{V_G\cup V_H}$.
By Proposition \ref{th:hammersleyclifford2}, $\nu^H$ is a \mrft, and so $\gmu^H[ A\, | \, \partial^2A] (\id_{\partial^2A}) = \gmu^H[ A \, | \, V_H \backslash A] (\id_{V_H \backslash A})$ a.s.  
Hence,
\begin{align*}
   \dfrac{ d\gmu^H[ A\, | \, \partial^2A] (\id_{\partial^2A})}{d\gmu^*[A]} (\id_A) 
 &=\dfrac{\frac{d\gmu^H}{d\gmu^*[V_H]}(\id_{V_H})}{\E\left[\frac{d\gmu^H}{d\gmu^*[V_H]} (\id_{V_H})\,| \,\id_{V_H \backslash A}\right]} \\
	&= \dfrac{\prod_{K \in \mathrm{cl}_2(H)}f^H_K(\id_K)}{\E[\prod_{K \in \mathrm{cl}_2(H)}f^H_K(\id_K) \, | \, \id_{V_H \backslash A}]},
\end{align*}
where we emphasize that the expectation in the denominator is with respect to independent random variables $(\id_v)_{v \in V_G \cup V_H}$.

The key observation is that if $K \in \mathrm{cl}_2(H)$ does not intersect $A$,
then the term $f_K^{H}(\id_K)$ factors out of the conditional expectation and cancels. 
Hence,  with $\mathcal{K}_A$ as in  \eqref{eq:cliques}, we see that 
\begin{align} 
 \dfrac{ d\gmu^H[ A\, | \, \partial^2A] (\id_{\partial^2A})}{d\gmu^*[A]} (\id_A) 
  &=  \frac{\prod_{K \in \mathcal{K}_A}f^H_K(\id_K)}{\E[\prod_{K \in \mathcal{K}_A}f^H_K(\id_K) \, | \, \id_{V_H \backslash A}]}. \label{pf:specification-abstract-1}
\end{align}
Since $\id_{V_H \setminus V^*}$ is independent of $\id_{V^*}$,  in view of \eqref{prop-cliques},
we may equivalently condition on $\id_{V^* \backslash A}$  in the denominator of the term on the right-hand side of  \eqref{pf:specification-abstract-1} to obtain 
\begin{align*}
   \dfrac{ d\gmu^H[ A\, | \, \partial^2A] (\id_{\partial^2A})}{d\gmu^*[A]} (\id_A) 
 &= \frac{\prod_{K \in \mathcal{K}_A}f_K^H(\id_K)}{\E[\prod_{K \in \mathcal{K}_A}f_K^H(\id_K) \, | \, \id_{V^* \backslash A}]}.
\end{align*}
Repeating the same arguments that led us to this point, we also find that   
\begin{align*}
   \dfrac{ d\gmu^G[ A\, | \, \partial^2A] (\id_{\partial^2A})}{d \gmu^*[A]} (\id_A) 
 &=  \frac{\prod_{K \in \mathcal{K}_A}f^G_K(\id_K)}{\E[\prod_{K \in \mathcal{K}_A}f^G_K(\id_K) \, | \, \id_{V^* \backslash A}]}.
\end{align*}
Recalling that $f^H_K \equiv f^G_K$ for $K \in \mathcal{K}_A$ by assumption, the proof is complete. 
\end{proof}

The last lemma allows us to deduce the following  insensitivity result that shows that given a finite graph $G = (V,E)$ and  associated  SDE \eqref{fingraph-SDE},
the conditional law of trajectories of particles
in a set $A \subset V$ given the trajectories of
particles at the double-boundary $\partial^2 A$ of the set does not depend on the graph structure outside of $A \cup \partial^2A$.

\begin{proposition} \label{pr:specification-finitegraph} 
Let $G=(V_G,E_G)$ and $H=(V_H,E_H)$ be finite graphs, and assume $V^* \subset V_G \cap V_H$ satisfies \eqref{asmp:edgesets}.
    Let $A \subset V^*$ satisfy $\partial^2_G A \subset V^*$ and $\partial^2_H A \subset V^*$, so that Lemma \ref{le:specification-abstract} ensures that $\partial^2_H A = \partial^2_G A =: \partial^2A$ and that \eqref{eq:cliques} holds (defining $\mathcal{K}_A$ as therein).
    Suppose   $(G, b^G, \sigma^G, \mu_0^G)$ and
    $(H, b^H, \sigma^H, \mu_0^H)$ both satisfy Assumption \ref{assumption:A}, and  let
    $P^G \in \P(\C^{V_G})$ and $P^H \in \P(\C^{V_H})$ be the corresponding unique laws of the  SDE described in \eqref{fingraph-SDE}. 
    Further, suppose the following consistency
    conditions hold:
\begin{enumerate}[(i)]
    \item We have
\begin{align}
b^H_v \equiv b^G_v, \quad \text{ for } v \in A \cup \partial^2 A,  \label{def:b-consistency} \\
\sigma^H_v \equiv \sigma^G_v, \quad \text{ for } v \in V_G \cap V_H. \label{def:sigma-consistency}
\end{align}
\item There is a product measure $\mu_0^* = \prod_{v \in V_G \cup V_H}\lambda_v \in \P((\R^d)^{V_G \cup V_H})$ for some $\lambda_v \in \P(\R^d), v \in V_G \cup V_H,$ such that $\mu^G_0$ and $\mu^H_0$ admit $2$-clique factorizations: 
  \begin{align}
    \label{mu0-factorization}
\frac{d\mu_0^G}{d\mu_0^*[V_G]}(x_{V_G}) = \prod_{K \in \mathrm{cl}_2(G)}f^G_K(x_K), \quad\quad \quad \frac{d\mu_0^H}{d\mu_0^*[V_H]}(x_{V_H}) = \prod_{K \in \mathrm{cl}_2(H)}f^H_K(x_K), 
  \end{align}
 for some measurable functions $(f^G_K: (\R^d)^K \mapsto \R_+)_{K \in \mathrm{cl}_2(G)}$ and  $(f^H_K: (\R^d)^K \mapsto \R_+)_{K \in \mathrm{cl}_2(H)}$
that satisfy the consistency condition $f^G_K \equiv f^H_K$ for every $K \in \mathcal{K}_A$.
\end{enumerate}
Then $P^G_t[A \, | \, \partial^2A] = P^H_t[A \, | \, \partial^2A]$ for each $t > 0$, both in the sense of $P^H_t[\partial^2A]$-almost sure and $P^G_t[\partial^2A]$-almost sure equality. 
\end{proposition}
\begin{proof}
As in \eqref{eq:canonical_law}, let $P^* \in \P(\C^{V_G \cup V_H})$ be the unique law of the solution $X=(X_v)_{v \in V_G \cup V_H}$ of the driftless SDE
\begin{align*}
dX_v(t) = \sigma_v^G(t,X_v)\,dW_v(t), \ \ v \in V_G, \qquad dX_v(t) = \sigma_v^H(t,X_v)\,dW_v(t), \ \ v \in V_H \setminus V_G,
\end{align*}
initialized with $X(0) \sim \mu_0^*$. 
Again working on the canonical probability space $(\C^{V_G \cup V_H},\text{Borel},P^*)$, define the martingales
$M^H = (M^H_v)_{v \in H}$: 
\[
M^H_v(t) := \int_0^t (\sigma_v^H(\sigma_v^H)^\top)^{-1}b^H_v(s,X_v,X_{N_v(H)})\cdot dX_v(s),  \quad v \in V_H,
\]
with $M^G = (M^G_v)_{v \in V_G}$, defined analogously, as in \eqref{martv}. 
Using \eqref{P-factorization} and  \eqref{mu0-factorization} we can write
\begin{align*}
  \frac{dP^H_t}{dP^*_t[V_H]}  &= \prod_{K \in \mathrm{cl}_2(H)}f^H_K(X_K(0))\,\prod_{v \in V_H}\EE_t(M^H_v),  \\
  \frac{dP_t^G}{dP^*_t[V_G]}  &= \prod_{K \in \mathrm{cl}_2(G)}f^G_K(X_K(0))\,\prod_{v \in V_G}\EE_t(M^G_v).
\end{align*}
with $\EE_t$ defined as in \eqref{def:doleans-exponential}.  
Note that if $v \in A \cup \partial A$, then we have $N_v(H)=N_v(G)$; 
indeed, this is due to \eqref{asmp:edgesets} and the inclusions 
$\partial_G A \subset V^*$ and $\partial_H A \subset V^*$.
Thus, by the consistency conditions \eqref{def:b-consistency} and \eqref{def:sigma-consistency}
along with the expressions above for $M^H$ and $M^G$,  we have 
$\EE_t(M^H_v) = \EE_t(M^G_v)$ for $v \in A \cup \partial A$.
Applying 
Lemma \ref{le:specification-abstract} with ${\mathcal X} = {\mathcal C}_t$,
$\nu^* = P^*$, $\nu^H = P^H_t$ and $\nu^G = P^G_t$, it follows from the consistency conditions
(i) and (ii) that 
 $P^G_t[A \, | \, \partial^2A] = P^H_t[A \, | \, \partial^2A]$
  holds in the sense of $P^*_t[\partial^2A]$-almost sure equality.
  Since both $P^H_t[\partial^2A]$ and $P^G_t[\partial^2A]$ are absolutely continuous with
  respect to $P^*_t[\partial^2A]$, the claim follows. 
\end{proof}

\section{Markov random field property for infinite-dimensional diffusions}
\label{sec-pf-infinitegraph} 

Fix a countably infinite connected graph $G = (V,E)$, and let
$(G,b,\sigma,\mu_0)$ be as in Assumption \ref{assumption:B}.  As usual, let
 $P = P^{\mu_0}$ and $P_t = P^{\mu_0}_t$ denote the unique law of the SDE \eqref{fingraph-SDE} and its projection, and let $P^{*,\mu_0}$ be the
 law of the canonical SDE system \eqref{eq:canonical_law} started from initial law $\mu_0$. 
 In this section we will prove Theorem \ref{th:conditionalindependence-infinitegraph}, that is, the \mrft \  property for $P_t$ and $P$.
 We will also use the same canonical space $(\C^V, \text{Borel}, P^{*,\mu_0})$ and 
canonical processes $(X_v)_{v \in V}$ as in  Section \ref{subs-approxmeas}.

 Throughout, choose an arbitrary vertex $\o$ in $V$ to be the root, and  
 let $G_n = (V_n, E_n)$ and $U_n = V_n\setminus V_{n-2}$  be as defined in Lemma \ref{le:MRFprojections}. 
 Also,  set 
 \begin{equation}
   \label{driftn}
   b_v^n = b_v,  \quad v \in V_{n-2}, \qquad  b_v^n = 0,  \quad v \in U_n,
   \end{equation}
(The family $\{b_v^n: n \geq 3, v \in U_n\}$ is arbitrary and set to zero for convenience, but more generally must merely be measurable and uniformly bounded.)
Let $\{P^n\}_{n \in \N}$ and
$\{P^n_t\}_{n \in \N}$ 
be the corresponding approximating sequence of measures and its projections, as defined in Section \ref{subs-approxmeas}.  
We first establish  
tightness and convergence results for $\{P^n_t\}_{n \in \N}$ in Section \ref{subs-tightness} and finally 
present the proof of Theorem \ref{th:conditionalindependence-infinitegraph} in Section \ref{subs-pf-infingraph}.

\subsection{Tightness and convergence results}
\label{subs-tightness}

In the following, let $H(\cdot \, | \, \cdot)$ denote relative entropy, defined for $\nu \ll \mu$ by
\[
H(\nu | \mu) = \int \frac{d\nu}{d\mu}\log \frac{d\nu}{d\mu}\,d\mu, 
\]
and $H(\nu|\mu) = \infty$ for $\nu \not\ll \mu$. Recall also our notation $\|x\|_{*,t} := \sup_{0 \le s \le t}|x_s|$ for the truncated supremum norm.

\begin{lemma} \label{le:tightness}
Suppose Assumption \ref{assumption:B} holds.  
For each $t > 0$ and each finite set $A \subset V$, we have  
\begin{align}
	\sup_n\sup_{v \in V_n}\E^{P^n}\left[\|X_v\|_{*,t}^2 \right] &< \infty, \label{def:secondmomentbound} \\
	\sup_n H\left(\left. P^n_t[A] \, \right| \, P^{*,\mu_0}_t[A]\right) &< \infty, \label{def:entropybound} \\
	\sup_n H\left(\left. P^{*,\mu_0}_t[A] \, \right| \, P^n_t[A]\right) &< \infty. \label{def:entropybound-2}
\end{align} 
\end{lemma}
\begin{proof}
Fix $t > 0$. 
We begin with  a standard estimate.  Recall the definition of
$b^n_v$ from \eqref{driftn},  
apply It\^{o}'s formula to the SDE \eqref{fingraph-SDEn}, and use the linear growth of $b_v$ from Assumption (\ref{assumption:B}.2) along with the uniform boundedness of $\sigma_v$ from Assumption 
(\ref{assumption:B}.3)  
to conclude that, for each $n \in \N$ and $v \in V_n$,  
\begin{align*}
  \E^{P^n} \left[ \|X_v\|_{*,t}^2 \right] & \le C\E^{P^n}
  \left[|X_v(0)|^2 +  \int_0^t |b^n_v(s,X_v,X_{N_v(G)})|^2 \,ds + \int_0^t |\sigma_v(s,X_v(s))|^2 \,ds \right] \\
	&\le C \E^{P^n} \left[1 + |X_v(0)|^2 + \int_0^t\left( \|X_v\|_{*,s}^2 + \frac{1}{|N_v(G)|}\sum_{u \in N_v(G)}\|X_u\|_{*,s}^2\right)ds\right],
\end{align*}
where $C < \infty$ is a constant that can change from line to line but does not depend on $n$ or $v$.
This implies that 
\begin{align*}
\sup_{v \in V_n}\E^{P^n}\left[ \|X_v\|_{*,t}^2 \right] &\le C\left( 1 + \sup_{v \in V} \E^{P^n}\left[|X_v(0)|^2\right] + \int_0^t\sup_{v \in V_n}\E^{P^n}\left[\|X_v\|_{*,s}^2 \right] ds\right), 
\end{align*}
where we have used the inclusion $V_n \subset V$. 
Apply Gronwall's inequality 
to find 
\begin{align}
\sup_{v \in V_n}\E^{P^n}\left[\|X_v\|_{*,t}^2 \right]\le C\left(1 + \sup_{v \in V}\int_{(\R^d)^V} |x_v|^2\,\mu_0(dx_V)\right). \label{pf:quadraticestimate}
\end{align}
The right-hand side is finite by Assumption (\ref{assumption:B}.1), and so    \eqref{def:secondmomentbound} follows.

Given a finite subset $A \subset V$, 
define $Q^n_t \in \P(\C_t^V)$ by
\begin{align*}
  \frac{dQ^n_t}{dP^{*,\mu_0}_t} = \prod_{v \in V_n \setminus A}\EE_t(  M_v^n ),
  \ \ \mbox{ where }  \ \ 
         M_v^n :=  \int_0^\cdot (\sigma_v\sigma_v^\top)^{-1}b^n_v(s,X_v,X_{N_v(G_n)}) \cdot dX_v(s), \ v \in V. 
\end{align*}
Due to Assumptions (\ref{assumption:B}.2) and (\ref{assumption:B}.3), and the definition of $P^{*,\mu_0}$ from Remark \ref{rem-driftless}, 
we can apply Lemma \ref{lem:Girsanov-justification} with $\QQ= P^{*,\mu_0}$, $X=(X_v)_{v \in V_n}$ and $f(t,x) = (\one_{\{v \in V_n \setminus A\}} \sigma_v^{-1}b^n_v(t,x_v,x_{N_v(G_n)}))_{v \in V_n}$, to conclude that $\frac{dQ^n_t}{dP^{*,\mu_0}_t}$ is a true $P^{*,\mu_0}$-martingale. 
It then follows from Girsanov's theorem \cite[Corollary 3.5.2]{karatzas-shreve}  and the uniqueness in law of the driftless SDE \eqref{eq:canonical_law-individual} that $Q^n_t[A] = P^{*,\mu_0}_t[A]$.  By a similar argument, 
\begin{align*}
 \frac{dP^n_t}{dP^{*,\mu_0}_t} = \prod_{v \in V_n} \EE_t( M_v^n ), \quad \ \text{and thus}  \ \quad   \frac{dP^n_t}{dQ^n_t} = \prod_{v \in A \cap V_n} \EE_t( M_v^n ), 
\end{align*}
where, in case $A \cap V_n = \emptyset$, we interpret the empty product as $1$. 
Once again invoking the linear growth of $b$, the boundedness of $\sigma_v$, 
   \eqref{pf:quadraticestimate}, the fact that
   $Q_t^n[A] = P^{*,\mu}[A]$ and Remark \ref{rem-driftless},  note that 
   Girsanov's theorem
   also shows that for every $v \in A \cap V_n$, under $P^n$, $M_v^n - [M_v^n]$  
   is a martingale and
   $[M_v^n](t) = 
 \int_0^t |\sigma_v^{-1} b_v^n (s, X_v, X_{N_v(G_n)})|^2 \, ds$.
Then, use the data processing inequality of relative entropy 
to obtain
\begin{align*}
	H(P^n_t[A] \, | \, P^{*,\mu_0}_t[A]) &= H(P^n_t[A] \, | \, Q^n_t[A]) \\
	&\le H(P^n_t \, | \, Q^n_t) \\
        & =   \sum_{v \in A \cap V_n} \E^{P^n} \left[ M_v^n(t) - \frac{1}{2} [M_v^n](t) \right] \\  
	&= \frac12 \sum_{v \in A \cap V_n} \E^{P^n}  \left[\int_0^t\, | \sigma_v^{-1}b^n_v(s,X_v,X_{N_v(G_n)})|^2\,ds \right]\\
	&\le C\sum_{v \in A \cap V_n}\E^{P^n}\left[1 + \|X_v\|_{*,t}^2 + \frac{1}{|N_v(G_n)|} \sum_{u \in N_v(G_n)} \|X_u\|_{*,t}^2\right] \\
	&\le C|A|\left(1 + 2 \sup_{v \in V_n}\E^{P^n} \left[\|X_v\|_{*,t}^2 \right]\right).
\end{align*}
Therefore \eqref{def:entropybound} follows from \eqref{def:secondmomentbound}.

Noting that due to the identity $Q_t^n[A] = P^{*,\mu}[A]$ and Remark \ref{rem-driftless}, 
 under $Q^n$,  $(X_v)_{v \in A}$ is driftless and 
  $M_v^n$ is a martingale. 
Therefore,  noting that
\begin{align*}
  \frac{dQ^n_t}{dP^n_t} = \prod_{v \in A \cap V_n}
  \exp \left( -M_v^n (t) 
  + \half \int_0^t |\sigma_v^{-1}b^n_v(s,X_v,X_{N_v(G_n)})|^2\,ds \right), 
\end{align*}
another application of the  the data processing inequality of relative entropy yields
\begin{align*}
	H(P^{*,\mu_0}_t[A] \, | \, P^n_t[A]) &= H(Q^n_t[A] \, | \, P^n_t[A]) \\
	&\le H(Q^n_t \, | \, P^n_t) \\
	&= \frac12 \sum_{v \in A \cap V_n} \E^{Q^n}  \left[\int_0^t| \sigma_v^{-1}b^n_v(s,X_v,X_{N_v(G_n)})|^2\,ds \right]\\
	&\le C\sum_{v \in A \cap V_n}\E^{Q^n}\left[1 + \|X_v\|_{*,t}^2 + \frac{1}{|N_v(G_n)|} \sum_{u \in N_v(G_n)} \|X_u\|_{*,t}^2\right] \\
	&\le C|A|\left(1 + 2 \sup_{v \in V_n}\E^{Q^n} \left[\|X_v\|_{*,t}^2 \right]\right).
\end{align*}
 The same argument that was used to obtain   \eqref{pf:quadraticestimate} 
     can  also be used to show that \eqref{pf:quadraticestimate} holds with $P^n$ replaced by $Q^n$. 
Therefore $\sup_n\sup_{v \in V_n}\E^{Q^n}\left[\|X_v\|_{*,t}^2 \right] < \infty$ by Assumption (\ref{assumption:B}.1), and hence the last display implies  \eqref{def:entropybound-2}. 
\end{proof}

The next lemma  will be used to show both that the \emph{existence} of a weak solution to the infinite SDE system \eqref{fingraph-SDE} holds automatically and also that it arises as the limit of finite-graph systems. Recall that $P \in \P(\C^V)$ denotes the law of the solution of \eqref{fingraph-SDE}.

\begin{lemma} \label{le:infinitegraphlimit}
Suppose Assumption \ref{assumption:B} holds.
Then $P^n \rightarrow P$ weakly on $\C^V$. Moreover, for any finite set $A' \subset V$, any $t > 0$, and any bounded measurable function $\psi : \C_t^{A'} \rightarrow \R$, we have 
\begin{align*}
\lim_{n\rightarrow\infty}\E^{P^n}[\psi(X_{A'}[t])] = \E^{P}[\psi(X_{A'}[t])].
\end{align*}
\end{lemma}
\begin{proof}
Fix $t > 0$.
The entropy bound of \eqref{def:entropybound} shows that $(P^n_t[A'])_{n \in \N}$ are precompact in the weak$^*$ topology induced on $\P(\C^{A'}_t)$ by the bounded measurable functions on $\C_t^{A'}$ \cite[Lemma 6.2.16]{dembozeitouni}.
In particular, this sequence is tight, and since this holds for every finite set $A'$ and every $t > 0$ we deduce that the entire sequence $(P^n)_{n \in \N}$ is tight in $\C^V$.
Note also that for sufficiently large $n$ it holds under $P^n$ that the processes
\begin{align}
\int_0^s \sigma_v^{-1}(r,X_v)\,dX_v(r) - \int_0^s\sigma_v^{-1}(r,X_v)b_v(r,X_v,X_{N_v(G)})\,dr ,   \ \ \ s \ge 0, \ \ \ v \in V_{n-2}, \label{pf:infinitegraphlimit1}
\end{align}
are independent standard Wiener processes, due to the consistency condition for the $b^n_v$'s and the identity $N_v(G_n)=N_v(G)$ valid for $v \in V_{n-2}$.

Now let $Q \in \P(\C^V)$ be any weak (in the usual sense) subsequential limit of $(P^n)_{n \in \N}$, with $P^{n_k} \rightarrow Q$ weakly. The aforementioned precompactness in the weak$^*$ topology implies that
\[
\lim_{k\rightarrow\infty}\E^{P^{n_k}}[\psi(X_{A'}[t])] = \E^{Q}[\psi(X_{A'}[t])],
\]
for any finite set $A' \subset V$, any $t > 0$, and any bounded measurable function $\psi$ on $\C_t^{A'}$.
We conclude that, under $Q$, the processes in \eqref{pf:infinitegraphlimit1} are independent Wiener processes, for $v \in V$. This shows that $Q$ is the law of a weak solution of the SDE system \eqref{fingraph-SDE}, which we know to be unique by assumption (\ref{assumption:B}.5). Hence, $Q=P$.
\end{proof}

\subsection{Proof of the second-order Markov random field property on the infinite graph}
\label{subs-pf-infingraph}

\begin{proof}[Proof of Theorem \ref{th:conditionalindependence-infinitegraph}] 
Fix $(G = (V, E), b, \sigma, \mu_0)$ and $X = (X_v)_{v \in V}$   as in the statement of the theorem.
For $n \geq 4$, consider the sequence of graphs $G_n = (V_n, E_n), n \in \N$ constructed from $G$ as
in Lemma  \ref{le:MRFprojections}.  
We first note that due to the fact that $\mu_0$ is a \mrft \  by Assumption (\ref{assumption:B}.1), 
part (iii) of Lemma \ref{le:MRFprojections}, with $\X = \R^d, \gmu = \mu_0, \gmu^* = \mu_0^*$,
ensures that $\mu_0[V_n]$ is a \mrft \   with respect to the graph $G_n$. Moreover, since $d\mu_0[V_n]/d\mu_0^*[V_n]$ is strictly positive by Assumption (\ref{assumption:B}.1),  Proposition \ref{th:hammersleyclifford2}  
 shows that $\mu_0[V_n]$ admits a $2$-clique factorization with respect to the product measure $\mu_0^*[V_n]$ for each $n$.
 Hence, $\mu_0[V_n]$ satisfies Assumption (\ref{assumption:A}.1), which when combined with the definition of $b^n=(b^n_v)_{v \in V_n}$ in \eqref{driftn} and
   the fact that $b, \sigma$ satisfy Assumptions (\ref{assumption:B}.2) and (\ref{assumption:B}.3), 
   shows that  $(G_n, b^n, (\sigma_v)_{v \in V_n}, \mu_0[V_n])$ satisfy Assumption \ref{assumption:A}.  
 Since $P^n[V_n]$ is the law of the SDE \eqref{fingraph-SDEn} on the finite graph $G_n$,  it is a \mrft \  by 
 Theorem \ref{pr:conditionalindependence-finitegraph}.

 Now, fix two finite sets $A, B \subset V$ with $B$ disjoint of $A \cup \partial^2A$, where throughout, we use $\partial^2$ to denote
$\partial_G^2$. 
Let $n_0$ denote the smallest integer greater than or equal to $4$ for which $A \cup \partial^2A \cup B \subset V_{n_0-3}$, and let $n \geq n_0$.
Then,  part (iv) of Lemma \ref{le:MRFprojections}, again with 
${\mathcal X} = \R^d$, $\nu  = \mu_0$, and $\nu^*  = \mu^*_0$, 
ensures that 
 $\mu_0[V_n]$ and $\mu_0[V_{n_0}]$ admit $2$-clique factorizations which are consistent 
 in the sense that the corresponding measurable functions $f_K^{G_n}$ and $f_K^{G_{n_0}}$
 agree for every $K \in \mathrm{cl}_2(G_{n_0})$ that intersects $A$ (equivalently, for
   every $K \in \mathrm{cl}_2(G)$ that intersects $A$).  
Since  $b_v^{n} = b_v^{n_0} = b_v$ for all $v \in A \cup \partial^2 A$ by \eqref{driftn}, and since $A \cup \partial^2_G A \subset V_{n_0-3}$,
we may apply Proposition \ref{pr:specification-finitegraph}, with $G = G_{n}$, $H = G_{n_0}$, $V^*=V_{n_0-3}$,  $\mu_0^{G_k} = \mu_0[V_k]$, and $(b^{G_k}_v,\sigma^{G_k}_v) = (b^k_v,\sigma_v)$ for $v \in G_k$ and $k \in \{n_0, n\}$, to deduce that  
$P^n_t[A \, | \, \partial^2A] = P_t^{n_0}[A \, | \, \partial^2A]$ for all $n \geq n_0$.  
 In other words, this implies that given a bounded continuous function  $f$,  there exists
 a measurable function $\varphi$ (that does not depend on $n$) such that
 \begin{equation}
   \label{phi-meas}
\varphi(X_{\partial^2A}[t]) = \E^{P^n}[f(X_A[t]) \, | \, X_{\partial^2A}[t]], \ \ \  P^n-a.s., \text{ for } n \ge n_0.
 \end{equation}

Now, fix additional bounded continuous functions $g,h$. For $t > 0$, taking the conditional expectation with respect to $X_{V_n \setminus A}[t]$ inside the expectation on the left-hand side below and using the \mrft \  property of $P^n$  we have
\begin{align*}
\E^{P^n}[f(X_A[t])g(X_{\partial^2A}[t])h(X_B[t])] = \E^{P^n}[\E^{P^n}[f(X_A[t]) \, | \, X_{\partial^2A}]g(X_{\partial^2A}[t])h(X_B[t])].
\end{align*}
When combined with \eqref{phi-meas}, this implies 
\begin{align*}
\E^{P^n}[f(X_A[t])g(X_{\partial^2A}[t])h(X_B[t])] = \E^{P^n}[\varphi(X_{\partial^2A}[t])g(X_{\partial^2A}[t])h(X_B[t])].
\end{align*}
Using the second part of Lemma \ref{le:infinitegraphlimit} for the finite set $A' = A \cup \partial^2 A \cup B$, and for both 
$\psi (y_{A'}) := f(y_A) g(y_{\partial^2A}) h(y_B)$,
and $\psi( y_{A'}) = \varphi(y_{\partial^2A}) g(y_{\partial^2 A}) h(y_B)$ 
for $y_{A'} \in \C_t^{A'}$, 
we may pass to the limit $n\rightarrow\infty$ and denote $P = P^{\mu_0}$ to get
\begin{align*}
\E^{P}[f(X_A[t])g(X_{\partial^2A}[t])h(X_B[t])] = \E^{P}[\varphi(X_{\partial^2A}[t])g(X_{\partial^2A}[t])h(X_B[t])].
\end{align*}
This at once shows both that
\[
\E^P[f(X_A[t]) \, | \, X_{\partial^2A}[t]] = \varphi(X_{\partial^2A}[t]) = \E^{P^n}[f(X_A[t]) \, | \, X_{\partial^2A}[t]],
\]
for all bounded continuous $f$ and $n \ge n_0$, which proves Proposition \ref{th:specification-infinitegraph} below,
and also that $X_A[t]$ and $X_B[t]$ are conditionally independent given $X_{\partial^2A}[t]$ under $P$. The latter proves the first statement in Theorem \ref{th:conditionalindependence-infinitegraph}, except for the fact that we have only proven this conditional independence when $A$ and $B$ are finite. 
Because $B \subset V \setminus (A \cup \partial^2A)$ was an arbitrary finite set and $\{X_B[t] : B \subset A \cup \partial^2A\}$ generates the same $\sigma$-field as $X_{A \cup \partial^2A}[t]$, we deduce that that $P_t$ is a \mrft. 
The second statement follows from the same argument as in \eqref{eq:infinitetime} using the \mrft \ property of $P_t = P_t^{\mu_0}$.
This completes the proof.
\end{proof}

We recapitulate two results that were established in the course of the proof, which may be of independent interest,  and which are used in the proof of Theorem \ref{th:gibbsuniqueness} in the next section. 

\begin{remark}
  \label{rem-pnmrf}
  Note that the first paragraph of the proof above shows that if $(G, b, \sigma, \mu_0)$ satisfy Assumption \ref{assumption:B} and
for $G_n = (V_n, E_n)$,   $n \in \N$, is as in Lemma  \ref{le:MRFprojections}, and 
  $b^n$, $P^n$, $P^n_t$, $n \in \N$, $t > 0$, are as defined at the beginning of Section \ref{sec-pf-infinitegraph},
  then $P^n_t[V_n]$ is a \mrft\  for each $n \in \N$ and $t > 0$. 
  \end{remark}

\begin{proposition} \label{th:specification-infinitegraph}
  Suppose $(G, b, \sigma, \mu_0)$ satisfy Assumption \ref{assumption:B}, and let $G_n = (V_n, E_n), n \in \N,$ be the sequence of
    graphs constructed from $G$ as
in Lemma  \ref{le:MRFprojections}. 
  Fix $t > 0$, and let $P_t = P_t^{\mu_0}$ be the law of
    the unique weak solution to the
    SDE \eqref{fingraph-SDE}  with initial law $\mu_0$.
  Then, for  $n \ge 3$, and $A \subset V_{n-3}$,
  $\partial^2_GA \subset V_n$ and, $P_t$-almost surely,
\[
P_t[A \, | \, \partial^2_G A] = P^n_t[A \, | \, \partial^2_G A].
\]
\end{proposition}

\section{Proof of  Gibbs measure properties} \label{se:gibbsuniqueness}

In this section we prove the Gibbs uniqueness property of Theorem \ref{th:gibbsuniqueness}.  
Recall the definition of $P^{*,\mu_0}$ as the law of the solution of \eqref{eq:canonical_law} initialized at $\mu_0$. \\

\noindent\textit{Proof of Theorem \ref{th:gibbsuniqueness}.}
Let $(G, b, \sigma, \mu_0)$ satisfy Assumption \ref{assumption:B}, and  let
    $P^{\mu_0}$ be the unique solution of the SDE system \eqref{fingraph-SDE} with initial law $\mu_0$. 
        We work again on the canonical space $(\C^V,\mathrm{Borel},P^{*,\mu_0})$, with
    $X_V=(X_v)_{v \in V}$ denoting the canonical process.
    Define the sets 
    $\MLset = \MLset (\mu_0)$ and $\MRset = \MRset (\mu_0)$ as in the
    statement of the theorem.
For any $\nu_0 \in \MLset$, 
 the SDE system \eqref{fingraph-SDE} is well-posed starting from $\nu_0$, 
and we let $P^{\nu_0} \in \P(\C^V)$ denote the law of this solution. 
The proof of the theorem is broken down into five claims. \\

\noindent
    {\em Claim 1. } Suppose $(G,b,\sigma,\mu_0)$   satisfies Assumption \ref{assumption:B}.  If 
      $\nu_0 \in \MLset$, then $(G,b,\sigma,\nu_0)$  also satisfies Assumption \ref{assumption:B}, and for every finite set $A \subset V$ and $t > 0$ we have $P^{\nu_0}_t [A] \sim P^{\mu_0}_t[A]$.
    \begin{proof}[Proof of Claim 1.]  
		First, suppose  $\nu_0 \in \MLset$.  
		By the definition of $\MLset$, $\nu_0$ has a finite  second moment.  Moreover, for each finite set $A \subset V$, we have $\nu_0[A] \sim \mu_0^*[A]$
		since $\nu_0 \in \G_2(\mu_0)$ implies (by Definition \ref{def-Gibbs}) that  $\nu_0[A] \sim \mu_0[A]$,  and Assumption (\ref{assumption:B}.1) ensures that 
		$\mu_0[A] \sim \mu_0^*[A]$.   Thus $\nu_0$ satisfies
		Assumption (\ref{assumption:B}.1).
		
		Now let $t > 0$, and let $A \subset V$ be finite. 
		From Lemma \ref{le:infinitegraphlimit} it follows that  $P^{n,\mu_0} \to P^{\mu_0}$ weakly.
		It then follows from \eqref{def:entropybound}, \eqref{def:entropybound-2} and the lower semicontinuity of relative entropy that $P^{\mu_0}_t[A] \ll  P^{*,\mu_0}_t[A]$ and $P^{*,\mu_0}_t[A] \ll P^{\mu_0}_t[A]$.
		Therefore $P^{\mu_0}_t[A] \sim  P^{*,\mu_0}_t[A]$, and similarly $P^{\nu_0}_t[A] \sim P^{*,\nu_0}_t[A]$. 
		Finally, note that $P^{*,\mu_0}[A]$ (resp.\ $P^{*,\nu_0}[A]$) is the law of the solution of the SDE system
		\[
		dX_v(t) = \sigma_v(t,X_v)\,dW_v(t), \qquad v \in A,
		\]
		with initial law $\mu_0[A]$ (resp.\ $\nu_0[A]$), and it follows from $\nu_0[A] \sim \mu_0[A]$ that $P_t^{*,\mu_0}[A] \sim P_t^{*,\nu_0}[A]$. Putting it together, we have 
		$P^{\nu_0}_t [A] \sim P^{*,\nu_0}_t[A] \sim P^{*,\mu_0}_t[A] \sim P^{\mu_0}_t[A]$. 
	\end{proof}

\noindent 
    {\em Claim 2. } For any $Q \in \MRset$, we have $Q_0 := Q\circ(X_V(0))^{-1} \in \MLset$.
\begin{proof}[Proof of Claim 2.]  The proof of this claim  is straightforward:
  fix  $Q  \in \MRset$, and set  
  $Q_0 := Q \circ (X_V(0))^{-1}$.
  Then by the definition of $\MRset$, we have
   $Q_t \in \G_2(P^{\mu_0}_t)$ for all $t \ge 0$ and
  $\sup_{v \in V} \int_{\R^d} |x_v|^2\,Q_0(dx) < \infty$.  Taking  
  $t=0$ gives $Q_0  \in \G_2(\mu_0)$,
  where we have used the elementary fact that $P^{\mu_0}_0= P^{\mu_0} \circ (X_V(0))^{-1}=\mu_0$. 
Thus $Q_0$ belongs to $\MLset$.
\end{proof}

\noindent 
    {\em Claim 3.} If $\nu_0 \in \MLset$ then
      $P^{\nu_0} \in \MRset$.
    \begin{proof}[Proof of Claim 3.] Fix $\nu_0 \in \MLset$.  Then
      by the first assertion of Claim 1, for every finite set $A \subset V$ and $t \geq 0$, $P^{\nu_0}_t[A] \sim P^{\mu_0}_t[A]$.     
  So it only remains to show that for every $t  > 0$, 
\begin{align}
  P^{\nu_0}_t[A \,|\, \partial^2 A] =  P^{\mu_0}_t[A \,|\, \partial^2 A],
  \quad \text{for finite } A \subset V. \label{pf:gibbsclaim0}
\end{align}

First, recall that Claim 1 also shows that  $(G,b,\sigma,\nu_0)$ satisfies Assumption \ref{assumption:B}. Next, let $G_n = (V_n, E_n)$ be the increasing sequence of finite graphs defined in 
  Lemma \ref{le:MRFprojections}, and  let $P^{\mu_0,n}, P^{\nu_0,n} \in \P(\C^V)$ denote the law of the solution of the corresponding SDE system \eqref{fingraph-SDEn} with initial laws
  $\mu_0[V_n]$ and $\nu_0[V_n]$, respectively.
Throughout this proof, the boundary operator $\partial$ is always with respect
to the infinite graph $G$.
  Fix $A \subset V$ finite, and fix $n$ large enough that $A \subset V_{n-3}$, recalling that $V_n$ was defined in Section \ref{subs-approxmeas}. 
By Proposition \ref{th:specification-infinitegraph}, we have both 
\begin{align}
\begin{split}
P^{\mu_0}_t[A | \partial^2 A] &=  P^{\mu_0,n}_t[A | \partial^2 A], \\
P^{\nu_0}_t[A | \partial^2 A] &=  P^{\nu_0,n}_t[A | \partial^2 A].
\end{split} \label{pf:gibbsclaim1}
\end{align}
By Lemma \ref{le:MRFprojections}(iii), $\nu_0[V_n]$ is a \mrft.
Also since $\nu_0[V_n] \sim \mu_0^*[V_n]$ implies $d\nu_0[V_n]/d\mu_0^*[V_n] > 0$,
by Proposition \ref{th:hammersleyclifford2} there is a 2-clique factorization
\begin{align}
\frac{d\nu_0[V_n]}{d\mu_0^*[V_n]}(x_{V_n}) = \prod_{K \in \mathrm{cl}_2(G_n)}g^n_K(x_K), \label{pf:gibbs1.5}
\end{align}
for some measurable functions $g^n_K : (\R^d)^K \to \R_+$. 
Similarly $\mu_0[V_n]$ admits a 2-clique factorization, 
\begin{align}
\frac{d\mu_0[V_n]}{d\mu_0^*[V_n]}(x_{V_n}) = \prod_{K \in \mathrm{cl}_2(G_n)}f^n_K(x_K), \label{pf:gibbs2}
\end{align}
for some measurable functions $f^n_K : (\R^d)^K \to \R_+$. We claim (and justify below) that $f^n_K$ and $g^n_K$ can be chosen to be consistent, i.e., so that
\begin{align}
f^n_K \equiv g^n_K, \quad \text{for all } K \in \mathrm{cl}_2(G_n) \text{ with } K \cap A \neq \emptyset. \label{pf:gibbsclaim3}
\end{align}
To see this, let $I_{V_n}=(I_v)_{v \in V_n}$ denote the canonical random variable on the probability space $(\R^d)^{V_n}$. Define
\[
\widehat{f}^n(I_{U_n}) = \E^{\mu_0^*}\left[\left.\frac{d\mu_0[V_n]}{d\mu_0^*[V_n]}\,\right|\,I_{U_n}\right], \qquad  \widehat{g}^n(I_{U_n}) = \E^{\mu_0^*}\left[\left.\frac{d\nu_0[V_n]}{d\mu_0^*[V_n]}\,\right|\,I_{U_n}\right].
\]
Recalling that $U_n = V_n \setminus V_{n-2}$, and using \eqref{pf:gibbs1.5}, we have
\begin{align*}
\frac{d\nu_0[V_{n-2} \, | \, U_n]}{d\mu_0^*[V_{n-2} \, | \, U_n]} &= \frac{ \frac{d\nu_0[V_n]}{d\mu_0^*[V_n]} }{ \widehat{g}^n(I_{U_n})} = \frac{1}{\widehat{g}^n(I_{U_n})}\prod_{K \in \mathrm{cl}_2(G_n)}g^n_K(I_K).
\end{align*}
Applying the same argument to $\mu_0$ rather than $\nu_0$ and using \eqref{pf:gibbs2}, we also obtain 
\begin{align*}
\frac{d\mu_0[V_{n-2} \, | \, U_n]}{d\mu_0^*[V_{n-2} \, | \, U_n]} &= \frac{ \frac{d\mu_0[V_n]}{d\mu_0^*[V_n]} }{ \widehat{f}^n(I_{U_n})} .
\end{align*}
Further, recognizing that $U_n = \partial^2 V_{n-2}$,
since $\nu_0 \in {\mathcal G}_2(\mu_0)$ we see that 
\[
\nu_0 [V_{n-2}| U_n] = \mu_0 [V_{n-2}| U_n].
\]
Combining the last three displays, we find
\begin{align*}
\frac{d\mu_0[V_n]}{d\mu_0^*[V_n]}(I_{V_n}) &= \frac{d\nu_0[V_{n-2} \, | \, U_n]}{d\mu_0^*[V_{n-2} \, | \, U_n]}\widehat{f}^n(I_{U_n}) = \frac{\widehat{f}^n(I_{U_n})}{\widehat{g}^n(I_{U_n})}\prod_{K \in \mathrm{cl}_2(G_n)}g^n_K(I_K).
\end{align*}
Comparing this with \eqref{pf:gibbs2} and noting that $U_n \in \mathrm{cl}_2(G_n)$, we can thus take $f^n_K \equiv g^n_K$ in \eqref{pf:gibbs2} for $K \in \mathrm{cl}_2(G_n) \setminus \{U_n\}$ and $f^n_{U_n} \equiv g^n_{U_n}\widehat{f}^n/\widehat{g}^n$. This proves the above consistency claim; indeed, since $A \subset V_{n-3}$, we know that $U_n$ does not intersect $A \cup \partial A$.

Let $\mathcal{K} = \{ K \in \mathrm{cl}_2(G_n) : K \cap A \neq \emptyset\}$. Using the consistency property \eqref{pf:gibbsclaim3}, we can finally conclude from Proposition \ref{pr:specification-finitegraph}  (with $V_G=V_H=V_n$, $b_v^H = b_v^G = b_v$, $\sigma_v^H = \sigma_v^H = \sigma_v$, $\mu_0^H = \mu_0$ and
$\mu_0^G = \nu_0$) that $P^{\nu_0,n}_t[A \,|\, \partial^2 A] =  P^{\mu_0,n}_t[A \,|\, \partial^2 A]$. Recalling \eqref{pf:gibbsclaim1}, this completes the proof of \eqref{pf:gibbsclaim0}.
\end{proof}

    Together, claims 2 and 3 prove \eqref{claim}. We now prove the last assertion of the theorem. \\

\noindent
    {\em Claim 4. } If $Q \in \P(\C^V)$ satisfies $Q_t \in \G_2(P^{\mu_0}_t)$ for all $t \ge 0$ and also $Q \circ (X_V(0))^{-1} = \mu_0$, then $Q=P^{\mu_0}$.
\begin{proof}[Proof of Claim 4.] 
      As in the proof of Claim 3, we  let $G_n = (V_n, E_n)$ be the increasing sequence of finite graphs defined in 
  Lemma \ref{le:MRFprojections} and  let the boundary operator $\partial$  always be with respect
to the infinite graph $G$.  Also, let $P^n = P^{\mu_0,n} \in \P(\C^V)$ denote the law of the solution of the corresponding SDE system \eqref{fingraph-SDEn} with initial law 
$\mu_0[V_n]$.
Now,  fix a finite set $A \subset V$ and $T \in (0,\infty)$. Let $n_0$ denote the smallest integer such that $A \cup \partial^2A \subset V_{n_0-3}$. Define the martingales
\begin{align*}
M^n_v(t) = \int_0^t  (\sigma_v\sigma_v^\top)^{-1}b^n_v(s,X_v,X_{N_v(G_n)}) \cdot dX_v(s),  \qquad n \in \N, v \in V_n.  
\end{align*}
Due to Assumptions (\ref{assumption:B}.2) and  (\ref{assumption:B}.3), 
it follows from  Lemma \ref{lem:Girsanov-justification} 
  (with $\QQ= P^{*,\mu_0}$, $X=(X_v)_{v \in V_n}$ and $f(t,x) = (\sigma_v^{-1}b^n_v(t,x_v,x_{N_v(G_n)}))_{v \in V_n}$) 
  that $\EE(M^n_v)$  is a $P^{*,\mu_0}$-martingale.  
Thus,  by Girsanov's theorem \cite[Corollary 3.5.2]{karatzas-shreve}, we may write
 \begin{align}
   \label{RNVn}
\frac{dP^n_t[V_n]}{dP^{*,\mu_0}_t[V_n]} &= \prod_{v \in V_n}\EE_t(M^n_v).  
\end{align}
Now, by applying Remark \ref{rem-pnmrf} to $(G,b,\sigma,\mu_0)$ and $(G,0,\sigma,\mu_0)$, respectively, it
follows that the measures $P^n_t[V_n]$ and  $P^{*,\mu_0}_t[V_n]$ are \mrft s  with respect to $G_n$.
Hence, for $n \ge n_0$, 
\begin{align*}
\frac{dP^n_t[A \, | \, \partial^2A]}{dP^{*,\mu_0}_t[A \, | \, \partial^2A]} &= \frac{dP^n_t[A \, | \, V_n \backslash A]}{dP^{*,\mu_0}_t[A \, | \, V_n \backslash A]}  \\
	&= \left.\frac{dP^n_t[V_n]}{dP^{*,\mu_0}_t[V_n]} \right/ \E^{P^{*,\mu_0}}\left[\left. \frac{dP^n_t[V_n]}{dP^{*,\mu_0}_t[V_n]} \right| X_{V_n \backslash A}[t]\right] \\
	&= \left.\prod_{v \in V_n}\EE_t(M^n_v)\right/ \E^{P^{*,\mu_0}}\left[\left.\prod_{v \in V_n}\EE_t(M^n_v)\right| X_{V_n \backslash A}[t]\right].
\end{align*}
For $v \in V_n \backslash (A \cup \partial A)$, $\EE_t(M^n_v)$ is measurable with respect to $X_{V_n \backslash A}[t]$ and thus factors out of the conditional expectation and cancels. Thus,
\begin{align}
  \label{RNcond}
\frac{dP^n_t[A \, | \, \partial^2A]}{dP^{*,\mu_0}_t[A \, | \, \partial^2A]} &= \left.\prod_{v \in A \cup \partial A}\EE_t(M^n_v)\right/ \E^{P^{*,\mu_0}}\left[\left.\prod_{v \in A \cup \partial A}\EE_t(M^n_v)\right| X_{V_n \backslash A}[t]\right].
\end{align}
Because $Q_t \in \G_2(P^{\mu_0}_t)$ by assumption, we have $Q_t[A \, | \, \partial^2A] = P^{\mu_0}_t[A \, | \, \partial^2A]$. By Proposition Proposition \ref{th:specification-infinitegraph}, we have $P^{\mu_0}_t[A \, | \, \partial^2A] = P^{\mu_0,n}_t[A \, | \, \partial^2A]$, and it follows that 
 the density $dQ_t[A \, | \, \partial^2A] / dP^{*,\mu_0}_t[A \, | \, \partial^2A]$ is given by the same expression \eqref{RNcond}.

Now take $A=V_{n-2}$, and note that $U_n := V_n \setminus V_{n-2} = \partial^2 V_{n-2}$.
Because $Q_t[U_n]  \sim P^{*,\mu_0}_t[U_n]$ by assumption, and because both $Q$ and $P^{*,\mu_0}$ start from the same initial state distribution $\mu_0$, we may use the martingale representation theorem (specifically, apply Remark \ref{re:mtgrep} below with $\xi= dQ_T[U_n]/dP^{*,\mu_0}_T[U_n]$,  which clearly satisfies $\E^{P^{*,\mu_0}}[\xi] = 1$) 
to find progressively measurable functions $r^n_v : [0,T] \times \C^{U_n} \to \R^d, v \in U_n,$ which are
  $dt\otimes dP^{*,\mu_0}$-square-integrable
  such that in terms of  the associated
   $X_{U_n}$-adapted continuous martingales 
  $R^n_v(t) = \int_0^t r^n_v(s,X_{U_n}) \cdot dX_v(s)$, $t \in [0,T]$, $v \in U_n$, we can write for $t \in [0,T]$,
  \begin{align*}
\frac{dQ_t[U_n]}{dP^{*,\mu_0}_t[U_n]} = \prod_{v \in U_n}\EE_t(R^n_v). 
\end{align*}
  Note that these martingales are orthogonal,   that is, the  covariation process $[R^n_v,R^n_u]$ is identically zero for $v \neq u$.
Thus, since $U_n \cap V_{n-2} = \emptyset$, $U_n \cup V_{n-2} = V_n$ and $V_{n-1} = V_{n-2} \cup \partial V_{n-2}$,  applying \eqref{RNcond} with $A = V_{n-2}$ we have
\begin{align*}
\frac{dQ_t[V_n]}{dP^{*,\mu_0}_t[V_n]} &=  \frac{dQ_t[V_{n-2} \, | \, U_n]}{dP^{*,\mu_0}_t[V_{n-2} \, | \, U_n]} \ \frac{dQ_t[U_n]}{dP^{*,\mu_0}_t[U_n]}  \\
	&= \prod_{v \in U_n}\EE_t(R^n_v)\left.\prod_{v \in V_{n-1}}\EE_t(M^n_v)\right/ \E^{P^{*,\mu_0}}\left[\left.\prod_{v \in V_{n-1}}\EE_t(M^n_v)\, \right| \, X_{U_n}[t]\right].
\end{align*}  
The process in the denominator is a positive martingale (as the optional projection of a martingale) adapted to $X_{U_n}$ and thus,
again using the  martingale representation theorem (this time applying Remark \ref{re:mtgrep} below with 
  $\xi = \E^{P^{*,\mu_0}}\left[\left.\prod_{v \in V_{n-1}}\EE_T(M^n_v) \, \right| \,  X_{U_n}[T]\right]$ and invoking \eqref{RNVn} to conclude that
  $\E^{P*,\mu_0}[\xi] = \E^{P*,\mu_0}[ dP_t^n[V_{n-1}]/dP_t^{*, \mu_0}[V_{n-1}]] = 1$),   
there exist square integrable, progressively measurable functions $\widetilde{r}^n_v$ (as above) and associated
$X_{U_n}$-adapted continuous martingales  $\widetilde{R}^n_v(t) = \int_0^t \widetilde{r}^n_v (s, X_{U_n})\cdot dX_v(s),$ 
$v \in U_n$, such that   
\begin{align*} 
\E^{P^{*,\mu_0}}\left[\left.\prod_{v \in V_{n-1}}\EE_t(M^n_v) \, \right| \,  X_{U_n}[t]\right] = \prod_{v \in U_n}\EE_t(\widetilde{R}^n_v),
\end{align*}
Now note that, for any continuous martingales $Z$ and $(Z_i)_{i \in I}$, with $I$ a finite index set, we have the identities $1/\EE(Z) = \EE(-Z)e^{[Z]}$ and
\begin{align*}
\prod_{i \in I} \EE(Z_i) &= \exp\left(\sum_{i \in I}Z_i - \frac12 \sum_{i \in I}[Z_i]\right) \\
	&= \exp\left(\sum_{i \in I}Z_i - \frac12 \Big[ \sum_{i \in I}Z_i \Big] + \frac12\sum_{i,j \in I, \, i\neq j}[Z_i,Z_j]\right) \\
	&= \EE\left(\sum_{i \in I}Z_i\right)\exp\left(\frac12\sum_{i,j \in I, \, i\neq j}[Z_i,Z_j]\right),
\end{align*}
where $[Z_i,Z_j]$ denotes the covariation process.
Hence,
\begin{align*}
\frac{dQ_t[V_n]}{dP^{*,\mu_0}_t[V_n]} &= \prod_{v \in U_n}\EE_t(R^n_v)\EE_t(-\widetilde{R}^n_v)\exp([\widetilde{R}^n_v](t))   \prod_{v \in V_{n-1}}\EE_t(M^n_v)  \\
	&= \prod_{v \in V_{n-2}}\EE_t(M^n_v)   \prod_{v \in V_{n-1} \setminus V_{n-2}}\EE_t(M^n_v + R^n_v - \widetilde{R}^n_v)\exp\left([M^n_v-\widetilde{R}^n_v,R^n_v-\widetilde{R}^n_v](t) \right)  \\
	&\quad \cdot \prod_{v \in V_n \setminus V_{n-1}}\EE_t(R^n_v - \widetilde{R}^n_v)\exp\left([\widetilde{R}^n_v,\widetilde{R}^n_v - R^n_v](t) \right).
\end{align*}
Recalling the orthogonality properties of $R^n_v$ and $\widetilde{R}^n_v$ mentioned above, we see that we can write $Z_t := dQ_t[V_n]/dP^{*,\mu_0}_t[V_n]$ in the form $Z(t) = \EE_t(N)e^{A(t)}$, where $N$ is a continuous square-integrable martingale and $A(t)$ is square-integrable and a.s.\ absolutely continuous with $A(0)=0$. Since $Z$ is a martingale, we necessarily have $A \equiv 0$; indeed, It\^o's formula gives $dZ(t)=Z(t)(dN(t) + dA(t))$, and for $Z$ to be a martingale we must have $dA(t)=0$.
It follows that 
\begin{align*}
\frac{dQ_t[V_n]}{dP^{*,\mu_0}_t[V_n]} &= \prod_{v \in V_{n-2}}\EE_t(M^n_v) \prod_{v \in V_{n-1} \setminus V_{n-2}}\EE_t(M^n_v + R^n_v - \widetilde{R}^n_v) \prod_{v \in V_n \setminus V_{n-1}}\EE_t(R^n_v - \widetilde{R}^n_v).
\end{align*} 
Since $Z$ is a $P^{*,\mu_0}$-martingale,
  Girsanov's theorem \cite[Corollary 3.5.2]{karatzas-shreve} can be applied, using
   the definition of $M_v^n$, to deduce that $Q_t[V_n]$ is the law of a solution $(X^n_v[t])_{v \in V_n}$ of the SDE system (perhaps on an auxiliary the probability space)
\begin{alignat*}{3}
dX^n_v(s) &= b_v(s,X^n_v,X^n_{N_v(G)})\,ds + \sigma(s,X^n_v)\,dB_v(s), &&\text{ for } v \in V_{n-2}, \\
dX^n_v(s) &= \left((r^n_v- \widetilde{r}^n_v)(s,X^n_{U_n}) + b^n_v(s,X^n_v,X^n_{N_v(G_n)})\right)ds + \sigma(s,X^n_v)\,dB_v(s), &&\text{ for } v \in V_{n-1} \backslash V_{n-2}, \\
dX^n_v(s) &= (r^n_v- \widetilde{r}^n_v)(s,X^n_{U_n})\,ds + \sigma(s,X^n_v)\,dB_v(s), &&\text{ for } v \in V_n \backslash V_{n-1},
\end{alignat*}
where $(B_v)_{v \in V_n}$ are independent Brownian motions.

Define $X^n_v \equiv 0$ for $v \notin V_n$. Since the sets $V_n$ increase to $V$, it is easily shown as in Lemma \ref{le:infinitegraphlimit} that, as $n \rightarrow \infty$, $(X^n_v[T])_{v \in V}$ converges in law in $\C_T^V$ to a solution of the infinite SDE system \eqref{fingraph-SDE} with initial distribution $\mu_0$, restricted to the interval $[0,T]$. Recalling that  $(X^n_v(0))_{v \in V} \sim \mu_0$ and that the solution to the infinite SDE system is  unique in law by Assumption (\ref{assumption:B}.4), we conclude that  $(X^n_v[T])_{v \in V}$ converges in law to $P_T^{\mu_0}$.  But $X^n_{V_n}[T]$ has law $Q_T[V_n]$ by construction, which implies $X^n_V[T]$ converges in law to $Q_T$. Therefore $Q_T = P_T^{\mu_0}$. 
Since $T \in (0,\infty)$ was arbitrary,  $Q= P^{\mu_0}$, which completes the proof of Claim 4.
\end{proof}

To complete the proof of the theorem, it only remains
to establish  the bijection between the two sets in \eqref{claim}.  However,
we now show that this is a simple consequence of the last claim. \\

    \noindent 
        {\em Claim 5. }
        The map $Q \mapsto Q \circ (X_V(0))^{-1}$ defines a bijection between  the sets 
    $\MRset$ and $\MLset$.
    \begin{proof}[Proof of Claim 5.] 
      Let $Q \in \MRset$, and set $\nu_0 := Q \circ (X_V(0))^{-1}$. By 
        Claim 2,  $\nu_0$ belongs to $\MLset$, 
        and by Claim 3,  $P^{\nu_0}$ lies in $\MRset$.  Since trivially 
        $P^{\nu_0} \circ (X_V(0))^{-1} = \nu_0$, to 
        prove the claim it suffices to prove that $Q = P^{\nu_0}$.
        By Claim 1,
        $(G,b,\sigma,\nu_0)$ satisfies Assumption \ref{assumption:B}, and thus Claim 4 applies with $\nu_0$ in place of $\mu_0$. That is, by applying Claim 4 to $\nu_0$ instead of $\mu_0$, we deduce that if $Q \in \P(\C^V)$ satisfies $Q_t \in \G_2(P^{\nu_0}_t)$ for all $t \ge 0$ and
        also $Q \circ (X_V(0))^{-1} = \nu_0$, then $Q=P^{\nu_0}$. By definition of $\MRset$ we have $Q_t \in \G_2(P^{\mu_0}_t)$ for all $t \ge 0$, and it follows from \eqref{pf:gibbsclaim0}, which was established in the proof of
        Claim 3, that $\G_2(P^{\nu_0}_t)=\G_2(P^{\mu_0}_t)$. This proves Claim 5, which completes the proof of Theorem \ref{th:gibbsuniqueness}.  
    \end{proof}

\begin{remark} \label{re:mtgrep} 
We sketch here the argument behind the use of the martingale representation theorem in the proof of Theorem \ref{th:gibbsuniqueness} above.
Recall that by Assumption (\ref{assumption:A}.3b) the SDE
system $dX_v(t) = \sigma_v(t,X_v)\,dW_v(t), \ v \in U_n$, with initial law $\mu_0$ is unique in law, with the law of the solution $X=(X_v)_{v \in U_n}$ given by $P^{*,\mu_0}[U_n]$.  
This implies uniqueness of the associated martingale problem (cf.\ \cite[Corollary 5.4.9]{karatzas-shreve}), which is known to imply that the solution has the predictable representation property (cf.\ \cite[Theorem V.25.1]{rogers-williams} or \cite[Theorem 2.7]{stroock1980extremal}), in the following sense: 
For $T < \infty$ and an $\F_T^X$-measurable random variable $\xi > 0$ with $\E[\xi]=1$, the martingale $Z(t)=\E[\xi \, | \, \F^X_t] > 0, t \in [0,T]$,  where $\F^X_t = \sigma(X(s):s \le t)$, can be represented as $Z(t) = 1 + \int_0^t\varphi(s,X) \cdot dX(s)$ for some predictable process $\varphi \colon [0,T] \times \C_T \to \Rmb$ satisfying $\int_0^T|\varphi(t,X)|^2\,dt < \infty$ a.s., recalling that $\sigma_v$ is uniformly bounded and nondegenerate. 
Then, for $t \in [0,T]$, setting $\psi(t,X)=\varphi(t,X)/Z(t)$, by It\^o's formula, we have $d\log Z(t) = \psi(t,X) \cdot dX(t) - \frac12\psi(t,X)^\top d[X](t)\psi(t,X)$.
Hence, $Z(t) = \EE_t(\int_0^\cdot \psi(s,X) \cdot dX(s))$, $t \in [0,T]$. 
\end{remark}

\appendix

\section{Proof of pathwise uniqueness under Lipschitz assumptions}
\label{ap:uniqueness-infSDE}

\begin{proof}[Proof of Proposition \ref{pr:uniqueness-infiniteSDE}]  
Let $(X_v)_{v \in V}$ and $(\widetilde{X}_v)_{v \in V}$ denote two solutions driven by the same Wiener processes and starting from the same initial states. 
Fix $T < \infty$.
For each $v \in V$ and $t \in [0,T]$,  by It\^{o}'s formula,  the boundedness
of $\sigma$ (see Assumption (\ref{assumption:B}.3)), the assumed Lipschitz condition on the drift and diffusion coefficients 
we have  
\begin{align*}
  \E \left[ \|X_v - \widetilde{X}_v\|_{*,t}^2 \right]& \le 2t
  \E \left[ \int_0^t \left| b_v(s,X_v(s),X_{N_v(G)}(s)) - b_v(s,\widetilde{X}_v(s),\widetilde{X}_{N_v(G)}(s))\right|^2 ds  \right] \\
	& \quad + 8\E \left[\int_0^t \left| \sigma_v(s,X_v(s)) - \sigma_v(s,\widetilde{X}_v(s))\right|^2 ds \right] \\
	& \le 4tK_T^2 \E \left[\int_0^t \left(\|X_v - \widetilde{X}_v\|_{*,s}^2 + \frac{1}{|N_v(G)|}\sum_{u \in N_v(G)} \|X_u - \widetilde{X}_u\|_{*,s}^2 \right) ds \right] \\
  & \quad + 8\bK_T^2\E \left[ \int_0^t \|X_v - \widetilde{X}_v\|_{*,s}^2 \,ds\right]. 
  \end{align*}
Hence,
\begin{align*}
\sup_{v \in V} \E \left[ \|X_v - \widetilde{X}_v\|_{*,t}^2 \right] & \le 8(tK_T^2+\bK_T^2) \int_0^t \sup_{v \in V} \E \left[ \|X_v - \widetilde{X}_v\|_{*,s}^2\right] ds.
\end{align*}
Complete the proof using Gronwall's inequality.
\end{proof}

\section{Justification for applying Girsanov's theorem}
\label{sec-Girsanov}

In this section we state a result that justifies our repeated application of Girsanov's theorem under the condition that the drift is progressively measurable and has  linear growth.
Lemma \ref{lem:Girsanov-justification} below is in fact a path-dependent multi-dimensional version of \cite[Theorems 5.1 and 8.1]{KlebanerLiptser2014when}.
A simpler proof is provided here for completeness.

Let $(\Omega,\F,\FF,\QQ)$ be a filtered probability space supporting a $\FF$-Brownian motion $W$ of dimension $m$ as well as an $\FF$-adapted process $X$ of dimension $d$ such that $X$ satisfies the SDE 
\begin{align}
	\label{eq:X-appendix}
	dX(t) = \sigma(t,X)\,dW(t), \quad X(0) \sim \mu, 
\end{align}
where $\mu \in \Pmc(\Rmb^d)$ and $\sigma : \R_+ \times \C \rightarrow \R^{d\times m}$ is bounded and progressively measurable.  Also, let $\E$ denote expectation with respect to $\QQ$. 
Fix a  progressively measurable $f: \R_+ \times \C \mapsto \R^m$, 
and define the stochastic integral 
\begin{equation*}
	M_t := \int_0^t f(s,X) \cdot dW_s, \qquad t \in [0, \infty), 
\end{equation*}
which is well defined (and a local martingale)
due to  the linear growth condition \eqref{linear-growth} imposed on $f$
in the lemma below.  Recall in what follows that $\|x\|_{*,t} = \sup_{s \in [0,t]} |x(s)|$. 

\begin{lemma}
	\label{lem:Girsanov-justification}
	Under the above setting, suppose for each $T \in (0,\infty)$
        there exists $C_T < \infty$ such that
	\begin{equation}
          \label{linear-growth}
	|f(t,x)| \le C_T\left(1 +  \|x(s)\|_{*,t} \right),
	\end{equation} 
	for all $t \in [0,T]$ and $x \in \C$.
	Then 
        the Doleans exponential $\{\EE_t(M)\}_{t \geq 0}$ defined in \eqref{def:doleans-exponential}
        is a true $\QQ$-martingale.
\end{lemma}

\begin{proof}
	Since $\{\EE_t(M)\}_{t \ge 0}$ is always a $\QQ$-supermartingale, it suffices to show that $\Emb [\EE_T(M)|X(0)=x]=1$ for each $x \in \Rmb^d$ and $T \in (0,\infty)$.
	So fix $T \in (0,\infty)$ and assume without loss of generality that $X(0)=x \in \Rmb^d$ in \eqref{eq:X-appendix}.  
	Since $\sigma$ is bounded, $X$ is a martingale and,
        from standard concentration inequalities for  martingales (see, e.g.,\ \cite[Lemma 2.1]{Vandegeer1995exponential}),  we can find some $C > 0$ such that
        $\QQ(\|X-x\|_{*,T} \ge a) \le \exp(-Ca^2)$ for each $a > 0$. 
	It then follows from the equivalence between sub-Gaussian tails and finite square-exponential moments (see, e.g., \cite[Section 2.3]{BoucheronLugosiMassart2013concentration}) that there exists $c > 0$ such that $\E[ \exp( c \|X\|_{*,T}^2 ) ] < \infty$. 	
	Now taking $0=t_0<t_1<\dotsb<t_{n(T)}=T$ with $t_n-t_{n-1} \le c/C_T^2$, and using the linear growth condition \eqref{linear-growth} on $f$, we have
	\begin{align*}
		\E \left[ \exp \left( \half \int_{t_{n-1}}^{t_n} |f(s,X)|^2 \,ds \right) \right] & \le \E \left[ \exp \left( (t_n-t_{n-1}) C_T^2(1+\|X\|_{*,T}^2) \right) \right] < \infty.
	\end{align*}
	It then follows from \cite[Corollary 3.5.14]{karatzas-shreve} that $\{\EE_t(M)\}_{t \geq 0}$ is a true $\QQ$-martingale.
\end{proof}

\bibliographystyle{amsplain}
\bibliography{condind}

\end{document}